\documentclass[reqno,xcolor=dvipsnames,dvipdfmx,11pt]{amsart}
\usepackage{amsaddr}
\usepackage{amsfonts}
\usepackage{amssymb}
\usepackage{amsthm}
\usepackage{amsmath}
\usepackage{mathrsfs}
\usepackage{lineno}
\newcommand*\patchAmsMathEnvironmentForLineno[1]{%
  \expandafter\let\csname old#1\expandafter\endcsname\csname #1\endcsname
  \expandafter\let\csname oldend#1\expandafter\endcsname\csname end#1\endcsname
  \renewenvironment{#1}%
     {\linenomath\csname old#1\endcsname}%
     {\csname oldend#1\endcsname\endlinenomath}}%
\newcommand*\patchBothAmsMathEnvironmentsForLineno[1]{%
  \patchAmsMathEnvironmentForLineno{#1}%
  \patchAmsMathEnvironmentForLineno{#1*}}%
\AtBeginDocument{%
\patchBothAmsMathEnvironmentsForLineno{equation}%
\patchBothAmsMathEnvironmentsForLineno{align}%
\patchBothAmsMathEnvironmentsForLineno{flalign}%
\patchBothAmsMathEnvironmentsForLineno{alignat}%
\patchBothAmsMathEnvironmentsForLineno{gather}%
\patchBothAmsMathEnvironmentsForLineno{multline}%
}
\usepackage{color}
\usepackage[]{natbib}
\usepackage[dvips]{graphicx}
\usepackage[cal=boondox]{mathalfa}
\usepackage[bbgreekl]{mathbbol}
\usepackage{mathtools}
\mathtoolsset{centercolon}
\newcommand{\blambda}{\overline{\lambda}}
\newcommand{\btlambda}{\overline{\lambda'}}
\newcommand{\tlambda}{\lambda'}
\newcommand{\hlambda}{\widehat{\lambda}}
\newcommand{\masslambda}{w}
\newcommand{\tmasslambda}{w'}
\newcommand{\bmu}{\overline{\mu}}
\newcommand{\dd}{\mathrm{d}}
\newcommand{\xx}{\boldsymbol{x}}
\newcommand{\yy}{\boldsymbol{y}}
\newcommand{\zz}{\boldsymbol{z}}
\newcommand{\E}{{\rm E}}
\newcommand{\talpha}{{\tilde{\alpha}}}
\newcommand{\tgamma}{{\tilde{\gamma}}}
\newcommand{\abslambda}{\lvert \lambda \rvert}
\newcommand{\absalpha}{\lvert \alpha \rvert}

\renewcommand{\Gamma}{\varGamma}
\newcommand{\xnum}{{N}}
\newcommand{\ynum}{{M}}
\newcommand{\Po}{\mathcal{\mathcal{P\hspace{-1.7pt}o}}}
\newcommand{\Ga}{\mathcal{G \hspace{-1.7pt} a}}
\newcommand{\Di}{\mathcal{D \hspace{-1.7pt} i}}
\newcommand{\NeBi}{\mathcal{N\hspace{-2.75pt}e\hspace{-1.7pt}B\hspace{-1.7pt}i}}

\newtheorem{theorem}{Theorem}

\newtheorem{lemma}{Lemma}

\allowdisplaybreaks[1]
\title[Shrinkage priors for nonhomogeneous Poisson processes]
{Shrinkage priors for nonparametric Bayesian prediction of nonhomogeneous Poisson processes}
\author{Fumiyasu Komaki$^{1,2,3}$}
\address
{$^1$Department of Mathematical Informatics,
The University of Tokyo, \\ 7-3-1 Hongo, Bunkyo-ku, Tokyo 113-8656, Japan.\\
$^2$RIKEN Center for Brain Science,
2-1 Hirosawa, Wako-shi, Saitama 351-0198, Japan. \\
$^3$International Research Center for Neurointelligence (IRCN), \\
The University of Tokyo, 7-3-1 Hongo Bunkyo-ku, Tokyo 113-0033, Japan.
}
\email{komaki@mist.i.u-tokyo.ac.jp}
\begin{document}


\maketitle
\begin{abstract}
We consider nonparametric Bayesian estimation and prediction
for nonhomogeneous Poisson process models with unknown intensity functions.
We propose a class of improper priors for intensity functions.
Nonparametric Bayesian inference with kernel mixture based on the class improper priors is shown to be useful,
although improper priors have not been widely used for nonparametric Bayes problems. 
Several theorems corresponding to those for finite-dimensional independent Poisson models
hold for nonhomogeneous Poisson process models with infinite-dimensional parameter spaces.
Bayesian estimation and prediction based on the improper priors
are shown to be admissible under the Kullback--Leibler loss.
Numerical methods for Bayesian inference based on the priors are investigated.
\end{abstract}
\section{Prediction based on models with finite dimensional parameter}
Statistical modeling and data analysis based on nonhomogeneous Poisson point processes have various applications
\citep[e.g.][]{SM91, Streit10}.
We consider nonparametric Bayesian inference with kernel mixture
for nonhomogeneous Poisson processes from the viewpoint of predictive density theory.
In the present paper,
it is shown that Bayesian procedures based on a class of improper priors provide reasonable estimation and prediction.
Several theorems concerning admissibility
corresponding to those for finite-dimensional independent Poisson models are shown to hold
for the nonhomogeneous Poisson process model.

In this section, we summarize the basic framework of predictive density theory mainly studied 
for finite-dimensional models.
Suppose that we have an observation $x$
from a probability density $p(x \mid \theta)$ that belongs to a parametric model
$\{p(x \mid \theta) \mid \theta \in \Theta \in \mathbb{R}^d \}$.
The objective is to predict an unobserved random variable $y$
distributed according to $p(y \mid \theta)$ using a predictive density $q(y;x)$.

We adopt the Kullback--Leibler loss
\begin{align*}
D\{p(y \mid \theta),q(y;x)\}=\int p(y \mid \theta)
\log \frac{ p(y \mid \theta)}{q(y;x)} \dd y.
\end{align*}
from the true density $p(y \mid \theta)$ to a predictive density $q(y;x)$.
A predictive density $q(y ; x)$ is said to dominate another predictive density $q'(y ; x)$
if the risk of $q(y ; x)$ is not greater than that of $q'(y ; x)$ for all $\theta$ and 
the strict inequality holds
for at least one point $\theta$ in the parameter space.
A predictive density $q(y ; x)$ is said to be admissible
if there does not exist a predictive density dominating $q(y;x)$.
See \citet{Berger85} and \citet{Schervish95}
for basic frameworks of statistical decision theory concerning admissibility and its relationship with Bayes theory.

Many studies recommend Bayesian predictive densities
\[
p_{\pi}(y \mid x)
= \frac{\int p(y \mid \theta)p(x \mid \theta)\pi(\theta) \dd \theta}
{\int p(x \mid \overline{\theta})\pi(\overline{\theta}) \dd \overline{\theta}}
\]
based on a prior $\pi(\theta)$ rather than plug-in densities
$p(y \mid \widehat{\theta}(x))$
\citep[e.g.][]{AD75,Geisseer93,K:Biometrika1996},
where $\widehat{\theta}$ is an estimated value of $\theta$.
This is because there exist Bayesian predictive densities
dominating plugin densities in many examples.

It is important to construct prior distributions for Bayesian inference
when we do not have specific prior information about unknown parameters.
Such a prior is called a noninformative prior or an objective prior.
The Jeffreys prior is a theoretically important objective prior.
Let $\theta = (\theta_1, \ldots, \theta_d)$ be an unknown parameter
of a finite dimensional parametric statistical model and $I(\theta)$ be the corresponding Fisher information matrix.
Then, the Jeffreys prior is defined by
\begin{align*}
 \pi_\mathrm{J}(\theta) \dd \theta_1 \cdots \dd \theta_d
= |I(\theta)|^{1/2} \dd \theta_1 \cdots \dd \theta_d,
\end{align*}
where $|I(\theta)|$ is the determinant of $I(\theta)$.
Noninformative priors such as the Jeffreys prior often becomes improper, i.e.\ the total mass of the prior is infinite.
The Bayes theorem cannot be directly applied to improper priors because they are not probability distributions, whose total masses are $1$.
However, formal application of the Bayes theorem to improper priors sometimes gives useful prediction and estimation.
Statistical decision theory concerning admissibility provides a basis of such generalized Bayesian procedure based on improper priors
\citep[e.g.][]{Berger85}.

Bayesian predictive densities based on shrinkage priors often dominate
the Bayesian predictive densities based on the Jeffreys prior
when the dimension of the parameter space is large.
Shrinkage priors assign more weight to parameter values close to
a subset in the parameter space than the Jeffreys prior.
In particular, Bayesian prediction based on shrinkage priors for finite-dimensional models such as
the multivariate Normal model
and the multidimensional independent Poisson model have been investigated.
See \citet{FSW:book18} for recent developments of parameter estimation theory based on shrinkage priors.

First, consider the $d$-dimensional Normal model.
Suppose that $x_i$ $(i=1,\ldots,d)$ are independently distributed
according to $\mbox{N}(\mu_i, \sigma^2)$
and that $y_i$ $(i=1,\ldots,d)$ are independently distributed
according to $\mbox{N}(\mu_i, \tau^2)$,
where $\mbox{N}(\mu_i,\sigma^2)$ is the Normal distribution with mean $\mu_i$ and variance $\sigma^2$.
Here, $\mu := (\mu_1, \ldots, \mu_d)$ is the unknown parameter and $\sigma$ and $\tau$ are known positive constants.
We consider prediction of $y = (y_1,y_2,\ldots,y_d)$
using $x=(x_1,x_2,\ldots,x_d)$
under the Kullback--Leibler loss.
The Bayesian predictive density $p_\mathrm{S}(y \mid x)$
based on the prior 
$\pi_\mathrm{S}(\mu) = \| \mu \|^{-(d-2)} = (\sum_i {\mu_i}^2)^{-(d-2)/2}$
introduced by \citet{Stein74}
dominates the Bayesian predictive density $p_\mathrm{J}(y \mid x)$
based on the Jeffreys prior $\pi_\mathrm{J}(\mu)=1$ \citep{K:Biometrika2001}.
This corresponds to the widely known result that
the generalized Bayes estimator based on Stein's prior $\pi_\mathrm{S}(\mu)$
dominates the best invariant estimator $\widehat{\mu} = x$ when $d \geq 3$.
See \cite{GLX06} for sufficient conditions for general priors other than the Stein prior
and \cite{BGX:AS08} for admissible predictive densities for Normal models.
The asymptotics of minimax risk of predictive density estimation for non-parametric regression
is studied by \citet{XuLiang10}.

Next, consider the $d$-dimensional Poisson model, which is closely related to the nonhomogeneous Poisson process models
considered herein.
Intuitively speaking, nonhomogeneous Poisson process models are infinite-dimensional Poisson models.
Suppose that $x_i$ $(i=1,\ldots,d)$ are independently distributed
according to $\mbox{Po}(s \lambda_i)$
and that $y_i$ $(i=1,\ldots,d)$ are independently distributed
according to $\mbox{Po}(t \lambda_i)$, where $\mbox{Po}(s \lambda_i)$ is the Poisson distribution with mean $s \lambda_i$,
$\lambda := (\lambda_1, \ldots, \lambda_d)$ is the unknown parameter, and $s$ and $t$ are known positive constants.
We consider prediction of $y = (y_1,y_2,\ldots,y_d)$ using $x=(x_1,x_2,\ldots,x_d)$
under the Kullback--Leibler loss.

A natural class of priors including the Jeffreys prior
$\pi_\mathrm{J} (\lambda) \dd \lambda_1 \dotsb \dd \lambda_d
= \lambda_1^{-\frac{1}{2}} \dotsb \lambda_d^{-\frac{1}{2}} \dd \lambda_1 \dotsb \dd \lambda_d$ is
$\pi_\alpha (\lambda) \dd \lambda_1 \dotsb \dd \lambda_d
:= \lambda_1^{\alpha_1-1} \dotsb \lambda_d^{\alpha_d-1} \dd \lambda_1 \dotsb \dd \lambda_d$,
where $\alpha_i > 0$ $(i=1,\ldots,d)$.
A class of improper prior densities
\begin{align}
\pi_{\alpha, \gamma}(\lambda) \dd \lambda_1 \dd \lambda_2 \cdots \dd \lambda_d
&= \frac{\lambda_1^{\alpha_1-1} \lambda_2^{\alpha_2-1} \cdots \lambda_d^{\alpha_d-1}}
{(\lambda_1+\lambda_2+\cdots+\lambda_d)^\gamma}
\dd \lambda_1 \dd \lambda_2 \cdots \dd \lambda_d
\label{eq:class}
\end{align}
with $\sum_i \alpha_i - \gamma > 0$ and $\alpha_i > 0$ $(i=1,2,\ldots,d)$ is investigated
and Theorems 1 and 2 below are obtained \citep{K:AS2004}.

\vspace{0.4cm}

\begin{theorem}[\citealp{K:AS2004}]
When 
$\sum_i \alpha_i - \gamma > 1 ~~\mbox{and}~~ \alpha_i > 0 ~~~ (i=1,2,\ldots,d)$,
the Bayesian predictive density
$p_{\alpha, \gamma}(y \mid x) ~~{\mbox{based on}}~~ \pi_{\alpha,\gamma}(\lambda)$
is dominated by the Bayesian predictive density
$p_{\talpha, \tilde{\gamma}}(y \mid x)~~\mbox{based on}~~ \pi_{\talpha, \tilde{\gamma}}(\lambda)$,
where $\tilde{\gamma} := \sum_i \alpha_i -1$ and
$\talpha = (\talpha_1, \talpha_2, \ldots, \talpha_d) := (\alpha_1, \alpha_2, \ldots, \alpha_d)$.
\end{theorem}

\begin{theorem}[\citealp{K:AS2004}]
For every $d \geq 1$, the Bayesian predictive densities based on the priors in the class
$\{\pi_{\alpha,\gamma}(\lambda): 0 < \sum_{i=1}^d \alpha_i - \gamma \leq 1, \alpha_i > 0~(i=1,2,\ldots,d)\}$
defined by (\ref{eq:class}) are admissible under the Kullback--Leibler loss.
\end{theorem}

\vspace{0.4cm}

In particular,
when $d \geq 3$,
the Bayesian predictive density $p_{\pi_\mathrm{S}}(y \mid x)$ based on the shrinkage prior
$\pi_\mathrm{S}(\lambda)
:= \pi_{\alpha = (\frac{1}{2},\dots,\frac{1}{2}),\, \gamma = \frac{d}{2}-1}(\lambda)$
dominates the Bayesian predictive density $p_{\pi_\mathrm{J}}(y \mid x)$ based on the Jeffreys prior
and is admissible under the Kullback--Leibler loss.
Parameter estimation
can be regarded as infinitesimal prediction under the Kullback--Leibler loss
in the multivariate Poisson model \citep{K:JMVA2006}.
Intuitively, infinitesimal prediction means prediction for infinitely near future.

In the present paper, we generalize the results for finite-dimensional independent Poisson models to the results for
nonhomogeneous Poisson models.
The remainder of the present paper paper is organized as follows.
In Section 2, a class of improper shrinkage priors for 
nonparametric Bayesian inference with kernel mixtures for nonhomogeneous Poisson models is introduced.
Several theorems concerning admissibility of Bayesian predictive densities and Bayes estimators
for nonhomogeneous Poisson models
corresponding to those for finite-dimensional models are proved.
In Section 3, numerical methods to evaluate Bayesian predictive densities and Bayes estimators are investigated.
Finally, conclusions are presented in Section 4.
\section{Nonparametric Bayesian inference for nonhomogeneous Poisson processes}
We consider nonhomogeneous Poisson processes on a region $U$ in the Euclidean space $\mathbb{R}^d$.
The results in the following can be generalized to those for nonhomogeneous Poisson processes
on general spaces such as a Polish space.

Basic properties of nonparametric inference of nonhomogeneous Poisson processes
using gamma process priors are given by \cite{Lo82} and \cite{LoWeng89}.
Corresponding results for probability density estimation are given by \cite{Lo84}.
Various estimation methods of intensity functions
of nonhomogeneous Poisson process models have been studied \citep[e.g.][]{BBMT81, Kolaczyk1999}.

Conventional statistical decision theory mainly deals with problems with finite dimensional parameters.
In the present paper, we show that the framework of statistical decision theory is effectively applied to
prediction and estimation for nonhomogeneous Poisson process models with infinite dimensional parameter spaces.
The results suggest that decision theoretic approach is also useful for other infinite dimensional problems.

\subsection{Bayes estimators and Bayesian predictive densities}

Let $\lambda(u)$ be a positive function of $u \in U$ satisfying $\int_U \lambda(u) \dd u < \infty$.
We observe $\xx = (\xnum, x_1, x_2, \ldots, x_\xnum)$
distributed according to the nonhomogeneous Poisson process $\Po(s\lambda)$ with intensity function $s \lambda$ ($s>0$).
Here, $\xnum$ is the number of the observed points and $x_1,\ldots,x_\xnum$ are observed points in $U$.
We introduce a known constant $s>0$ for later use.
Let $X(B) := \# \{ x_i \mid x_i \in B\ (1 \leq i \leq N)  \}$ and $\lambda(B) := \int_{B} \lambda(u) \dd u$
for $B \subset U$.
If $X(B_i)$  $(i=1,\ldots,k)$ are independently distributed according to the Poisson distribution with mean
$s\lambda(B_i)$ for an arbitrary partition $(B_1,\ldots,B_k)$ of $U$,
$\xx$ is said to be distributed according to the nonhomogeneous Poisson distribution
with intensity function $s \lambda$.
The function $\lambda$ is the unknown parameter.

The likelihood for the nonhomogeneous Poisson process model is given by
\begin{align}
\label{likelihood}
\biggl\{ \prod_{i=1}^\xnum & s \lambda(x_i) \biggr\}
\exp \biggl\{ - s \int \lambda(x) \dd x \biggr\}
=
\bigl\{ \prod_{i=1}^\xnum \blambda(x_i) \bigr\}
(s \masslambda)^\xnum \exp \bigl( - s \masslambda \bigr)  \\
&\propto
\biggl\{ \prod_{i=1}^\xnum \blambda(x_i) \biggr\}
\frac{(s \masslambda)^\xnum}{\xnum!} \exp \bigl( - s \masslambda \bigr)
\label{ldecomp}
=: p_\lambda(\xx),
\end{align}
where
$w = |\lambda| := \int_U \lambda(u) \dd u$ and $\overline{\lambda} := \lambda/w$
\citep[see e.g.][p.\ 22]{DV:book03}.
We identify $\lambda$ with $(\masslambda, \blambda)$.
The probability density \eqref{ldecomp} of the observed points
$(x_1,\ldots,x_\xnum)$ multiplied by $\xnum!$ coincides with \eqref{likelihood}.
These two representations do not make essential differences in the following discussions.
We mainly use \eqref{ldecomp} in the following.

Let $\yy = (\ynum, y_1, y_2, \ldots, y_\ynum)$ be
a sample independent of $\xx$ from the nonhomogeneous Poisson process $\mathcal{P\hspace{-1.7pt}o}(t \lambda)$,
where $t>0$ is a known constant.
We investigate estimation of $\lambda$ and prediction of $\yy$ using $\xx$.
In Subsection \ref{infinitesimal prediction}, estimation is formulated as a limit of prediction.

First, we consider estimation of $\lambda$.
From \eqref{ldecomp},
the Kullback--Leibler divergence from the probability density $p_\lambda(\yy)$
corresponding to the intensity $t\lambda$
to another probability density  $p_{\lambda'}(\yy)$
corresponding to the intensity $t\lambda'$ is
\begin{align}
D( p_\lambda(\yy), p_{\lambda'}(\yy))
=&\ \E_\lambda \log
\frac{\biggl\{\displaystyle \prod_{i=1}^\ynum \blambda(y_i) \biggr\}
\frac{(t \masslambda)^\ynum}{\ynum!} \exp \bigl( -t \masslambda \bigr)}
{\displaystyle \biggl\{ \prod_{i=1}^\ynum \overline{\lambda'}(y_i) \biggr\}
\frac{(t \masslambda')^\ynum}{\ynum!} \exp \bigl( -t \masslambda' \bigr)} \notag \\
=&\
\E_\lambda \left[ t \masslambda' - t \masslambda + \ynum \log \frac{\masslambda}{\masslambda'}
+ \sum_{i=1}^\ynum \log \frac{\blambda(y_i)}{\overline{\lambda'}(y_i)} \right] \notag \\
=&\
\label{eq:KLeat0}
t \masslambda \left( \frac{\masslambda'}{\masslambda} - 1 - \log \frac{\masslambda'}{\masslambda}
+ \int \blambda(y) \log \frac{\blambda(y)}{\overline{\lambda'}(y)} \dd y \right) \\
=&\ t \int
\biggl(\lambda'(y) - \lambda(y)
+ \lambda(y) \log \frac{\lambda(y)}{\lambda'(y)} \biggr) \dd y.
\label{eq:KLeat}
\end{align}

Suppose that a prior density $\pi(\dd \lambda)$ is adopted and observation $\xx$ is given.
Let $p_\pi(\dd \lambda \mid \boldsymbol{x})$ be the posterior distribution.
If the posterior mean of $\lambda$ has a density $\lambda_{\pi,\xx}(u)$ with respect to the Lebesgue measure on $U$,
the posterior mean of the Kullback--Leibler loss of an intensity estimator $\widehat{\lambda}$ is
\begin{align*}
\int
D( p_\lambda(\yy), p_{\lambda'}(\yy))\, p_\pi(\dd \lambda \mid \xx)
&= t \iint
\biggl(\hlambda(y) - \lambda(y)
+ \lambda(y) \log \frac{\lambda(y)}{\hlambda(y)} \biggr) \dd y \, p_\pi(\dd \lambda \mid \boldsymbol{x}) \\
&= t \iint \biggl\{
\biggl(\widehat{\lambda}(y) - \lambda_{\pi,\xx}(y) + \lambda_{\pi,\xx}(y)
\log \frac{\lambda_{\pi,\xx}(y)}{\widehat{\lambda}(y)}\biggr) \\
&\;\;\;\;\;\;+ \biggl(\lambda_{\pi,\xx}(y) - \lambda(y) + \lambda(y) \log \frac{\lambda(y)}{\lambda_{\pi,\xx}(y)}\biggr)
\biggr\} p_\pi(\dd \lambda \mid \boldsymbol{x}) \dd y
\end{align*}
and is minimized when $\widehat{\lambda} = \lambda_{\pi,\xx}$.
Here, we assume that the integral exists.
Thus, the Bayes estimator of $\lambda$
is the posterior mean $\lambda_{\pi,\xx}$ given observation $\xx$.

If the posterior is decomposed as $p_\pi(\dd w \, \dd \blambda \mid \xx) = p_\pi(\dd w \mid \xx) p_\pi(\dd \blambda \mid \xx)$,
then the Bayes estimators of $w$, $\blambda$, and $\lambda$ based on the prior $\pi$ are given by
\begin{align}
w_{\pi,\xx} &= \int w \, p_\pi(\dd w \mid \xx),\ \ 
\blambda_{\pi,\xx} = \int \blambda \, p_\pi(\dd \blambda \mid \xx), \ \mathrm{and} \ \
\lambda_{\pi,\xx} = w_{\pi,\xx} \blambda_{\pi,\xx}, \notag
\end{align}
respectively.

Next, we consider predictive densities of $\yy$.
The Kullback--Leibler divergence from the probability density $p_\lambda(\yy)$
corresponding to the intensity $t\lambda$
to another probability density $q(\yy) = q(\ynum,y_1,\ldots,y_\ynum) = q(\ynum)q(y_1,\ldots,y_\ynum \mid \ynum)$ is
\begin{align}
D(p_\lambda(\yy),q(\yy))
&= \E_\lambda \left[ \log \frac
{p^\mathrm{Po}_{t\masslambda}(\ynum) \prod_{i=1}^\ynum \overline{\lambda}(y_i)}
{q(\ynum, y_1,\ldots,y_\ynum)} \right],
\end{align}
where $p^\mathrm{Po}_{t\masslambda}(\ynum) := \{(tw)^M/M!\}\exp(-tw)$ 
denotes the Poisson probability density with mean $t\masslambda$.

Since the posterior mean of the Kullback--Leibler divergence is
\begin{align}
&\int D(p_\lambda(\yy),q(\yy;\xx)) p_\pi(\dd \lambda \mid \xx) \notag \\
&=
\iint \sum_{\ynum=0}^\infty \idotsint p^\mathrm{Po}_{t \masslambda}(\ynum)
\Bigl\{\prod_{i=1}^\ynum \overline{\lambda}(y_i)\Bigr\}
\log \frac{p^\mathrm{Po}_{t \masslambda}(\ynum) \prod_{i=1}^\ynum \overline{\lambda}(y_i)}
{q(\ynum,y_1,\ldots,y_\ynum)} \dd y_1 \cdots \dd y_\ynum
p_\pi(\dd w \, \dd \blambda \mid \xx),
\label{pm-KL}
\end{align}
the Bayesian predictive density minimizing \eqref{pm-KL} is
\begin{align}
p_\pi(\ynum,y_1,\ldots,y_\ynum \mid \xx) =
\iint p^\mathrm{Po}_{t \masslambda}(\ynum) \Bigl\{\prod_{i=1}^\ynum \overline{\lambda}(y_i)\Bigr\}
p_\pi(\dd w \, \dd \blambda \mid \xx). \notag
\end{align}

If the posterior is decomposed as $p_\pi(\dd w \, \dd \blambda \mid \xx) = p_\pi(\dd w \mid \xx) p_\pi(\dd \blambda \mid \xx)$,
then the Bayesian predictive density based on the prior $\pi$ is given by
\begin{align}
p_\pi(\ynum,y_1,\ldots,y_\ynum \mid \xx) = p_\pi(\ynum \mid \xx) p_\pi(y_1,\ldots,y_\ynum \mid \ynum, \xx),
\end{align}
where
\begin{align}
p_\pi(\ynum \mid \xx)
&= \int p^\mathrm{Po}_{t \masslambda}(\ynum) p_\pi(\dd w \mid \xx) \notag
\end{align}
and
\begin{align}
p_\pi(y_1,\ldots,y_\ynum \mid \ynum, \xx) =
\int \Bigl\{\prod_{i=1}^\ynum \overline{\lambda}(y_i)\Bigr\}
p_\pi(\dd \blambda \mid \xx). \notag
\end{align}
\subsection{Kernel mixture models}\label{Kernel}
We need to adopt a prior for $\lambda$ in order to use Bayesian methods.
First, we consider the gamma process prior $\Ga(\alpha, \beta)$,
where $\alpha(u)$ is a positive density function of $u$ and $\beta$ is a positive scalar that does not depend on $u$.
Then, for an arbitrary partition $(B_1,\ldots,B_k)$ of $U$,
$\mu_i := \mu(B_i)$ $(i=1,\ldots,k)$ are independently distributed according to the gamma distributions
$\mathrm{Ga}(\alpha_i,\beta)$ with densities
$\{1/\Gamma(\alpha_i)\} (\mu_i^{\alpha_i-1}/\beta^{\alpha_i}) \exp(-\mu_i/\beta)$,
where $\alpha_i := \alpha(B_i) = \int_{B_i} \alpha(u) \dd u$.
The mixture of the nonhomogeneous Poisson process $\Po(t \lambda)$
with respect to the prior
$\mathcal{G \hspace{-1.7pt} a}(\alpha, \beta)$ for $\lambda$ is 
the negative binomial process $\mathcal{N\hspace{-2.75pt}e\hspace{-1.7pt}B\hspace{-1.7pt}i}(\alpha, t \beta/(1 + t \beta))$.
For $\mathcal{N\hspace{-2.75pt}e\hspace{-1.7pt}B\hspace{-1.7pt}i}(\alpha, t \beta/(1 + t \beta))$
and an arbitrary partition $(B_1,\ldots,B_k)$ of $U$,
the numbers $N_i$ of points in $B_i$ $(i=1,\ldots,k)$ are independently distributed according
to the negative binomial distribution with density ${N_i + \alpha_i - 1 \choose \alpha_i}
 \{t \beta/(1 + t \beta)\}^{N_i} \{1/(1 + t \beta)\}^{\alpha_i}$.
The posterior with respect to the prior $\mathcal{G \hspace{-1.7pt} a}(\alpha, \beta)$ and observation $\xx$ is
$\Ga(\alpha + \sum_i \delta_{x_i}, 1/(s+1/\beta))$ \citep[see][]{LoWeng89}.
Then, the posterior means of $w$, $\blambda$, and $\lambda$ are
\begin{align*}
 w_{\alpha,\beta,\xx} = \frac{\beta}{1+s\beta} (|\alpha| + \xnum),~
 \blambda_{\alpha,\beta,\xx} = \frac{\alpha + \sum_i \delta_{x_i}}{|\alpha| + \xnum},~
\mbox{~and~~~}~
 \lambda_{\alpha,\beta,\xx} = \frac{\beta}{1+s\beta} (\alpha + \sum_i \delta_{x_i}),
\end{align*}
respectively,
where $\delta_{x_i}$ denotes the Dirac measure at $x_i$.
Thus, the Bayesian predictive density based on the gamma process prior $\Ga(\alpha, \beta)$ and observation $\xx$ is
\begin{align*}
\NeBi \biggl( \alpha + \sum_i \delta_{x_i}, \frac{(s + 1/\beta)^{-1} t}{1 + (s + 1/\beta)^{-1} t} \biggr)
&= \NeBi \biggl( \alpha + \sum_i \delta_{x_i}, \frac{t \beta}{ 1 + (s + t) \beta} \biggr).
\end{align*}

Although the gamma process prior is a conjugate prior for the nonhomogeneous Poisson model,
it is not natural to use this prior directly for $\lambda$
because the measure $\alpha + \sum_i \delta_{x_i}$ is not absolutely continuous with respect to the intensity measure $\lambda$.
In fact,
the Kullback--Leibler loss of the posterior mean $\lambda_{\alpha,\beta,\xx}$ with respect to
observation $\xx$ and the gamma prior $\Ga(\alpha,\beta)$ is
\begin{align*}
D(t \lambda, t \lambda_{\alpha,\beta,\xx})
=&\
t \masslambda \left\{ \frac{\masslambda_{\alpha,\beta}}{\masslambda} - 1 - \log \frac{\masslambda_{\alpha,\beta}}{\masslambda}
+ \int \blambda(y) \log \frac{\blambda(y)}{\overline{\lambda}_{\alpha,\beta,\xx}(y)} \dd y \right\} \\
=&\
t \masslambda \left\{ \frac{\masslambda_{\alpha,\beta}}{\masslambda} - 1 - \log \frac{\masslambda_{\alpha,\beta}}{\masslambda}
+ \int \blambda(y) \log \frac{\blambda(y)}{\overline{\alpha}(y)} \dd y + \log \biggl( 1+ \frac{\xnum}{|\alpha|} \biggr) \right\},
\end{align*}
where $|\alpha| := \int_U \alpha(u) \dd u$ and $\overline{\alpha} := \alpha/|\alpha|$.
The amount of information concerning $\lambda$ included in $\xx$ increases as $s$ increases.
However, when $s$ goes to infinity
and $t$ is fixed, $\xnum$ and the divergence $D(t \lambda, t \lambda_{\alpha,\beta})$ diverges to infinity.
Here, $w_{\alpha, \beta}$ converges to $w$ with probability $1$.
However, $\overline{\lambda}_{\alpha,\beta,\xx}$ does not converge to $\blambda$ in the Kullback--Leibler sense,
although $\overline{\lambda}_{\alpha,\beta,\xx}$ weakly converges to $\blambda$.

In order to overcome the difficulty caused by the fact that $\alpha + \sum_i \delta_i$ is not absolutely continuous
with respect to $\lambda$, kernel mixture models are widely used.
A nonnegative function $k: U \times U \rightarrow \mathbb{R}_{\geq 0}$
satisfying $\int k(y,u) \dd y = 1$ is called a kernel function.
Suppose that an intensity function $\lambda$ is represented by
\begin{align*}
\lambda(y) &= k(y,\mu) := \int k(y,u) \mu(\dd u),
\end{align*}
where $\mu(\dd u)$ is a finite measure on $U$.
Then,
\begin{align*}
\masslambda &= \abslambda
:= \int \lambda(y) \dd y = \iint k(y,u) \mu(\dd u) \dd y = \int \mu(\dd y) = |\mu|.
\end{align*}
We identify $\mu$ with $(\masslambda, \bmu)$, where $\overline{\mu} := \mu/|\mu|$.
Then, we have
\begin{align*}
\blambda(y) &= \frac{\lambda(y)}{\abslambda} = k(y,\overline{\mu}) = \int k(y,u) \overline{\mu}(\dd u).
\end{align*}

We hereinafter consider kernel mixture models because these models are reasonable under the Kullback--Leibler loss.
The results in the following can be generalized to the setting in which the kernel function has an unknown parameter
by introducing a prior for the unknown parameter.

\vspace{0.4cm}

\noindent
Example. The Gaussian kernel
\begin{align*}
 k(y,u) = \frac{1}{\sqrt{2 \pi \sigma^2}} \exp \biggl\{ - \frac{1}{2 \sigma^2} (y-u)^2 \biggr\} \hspace{8mm} (y,u \in \mathbb{R})
\end{align*}
is frequently used in applications.
For simplicity, we assume that $\sigma > 0$ is known.
\qed

\vspace{0.4cm}

Assume a gamma process prior $\Ga(\alpha,\beta)$ for $\mu$.
Then, $\overline{\mu} := \mu/\masslambda$ is distributed according to the Dirichlet process $\Di(\alpha)$
\citep[see e.g.][p.~96]{GR03}.
For every partition $(B_1,\ldots,B_k)$ of $U$,
$(\overline{\mu}(B_1),\ldots,\overline{\mu}(B_k))$
is distributed according to the $k$-dimensional Dirichlet distribution $\mathrm{Di}(\alpha(B_1),\ldots,\alpha(B_k))$.
The weight parameter $\masslambda$ is distributed according to the gamma distribution
$\mathrm{Ga}(|\alpha|,\beta)$ independently of $\overline{\mu}$.
Thus, the simultaneous distribution of $w$ and $\overline{\mu}$ is given by
\begin{align}
 p^{\Di}_\alpha(\dd \overline{\mu}) \, p^\mathrm{Ga}_{|\alpha|,\beta}(\masslambda) \, \dd \masslambda,
\label{gammaprior}
\end{align}
where $p^{\Di}_\alpha(\dd \overline{\mu})$ denotes the Dirichlet process measure and
$p^\mathrm{Ga}_{|\alpha|,\beta}(\masslambda) = \{1/\Gamma(|\alpha|)\} (w^{|\alpha|-1}/\beta^{|\alpha|}) \exp (-w/\beta)$
is the gamma probability density.

\subsection{Improper priors}
It is difficult to determine the scale parameter $\beta$ from the viewpoint of objective Bayes.
Let $c$ be an arbitrary positive constant for time scale change.
Then, $\Po(s\lambda) = \Po(\tilde{s}\tilde{\lambda})$,
where $\tilde{s} := cs$ and $\widetilde{\lambda} := \lambda/c$.
Inference for $\lambda$ is equivalent to inference for $\widetilde{\lambda}$
because the time scale change does not affect the essence of the problem.
Thus, the objective prior should be (relatively) invariant with respect to the time scale change.
However, the gamma process prior $\Ga(\alpha,\beta)$ is not relatively invariant if $\beta$ is finite.
One method by which to construct an invariant posterior is to adopt the improper prior
\begin{align}
\pi_\alpha(\mu) \dd \mu = p^{\Di}_\alpha(\dd \overline{\mu}) \, p^\mathrm{Ga}_{|\alpha|,\infty}(\masslambda) \, \dd \masslambda,
\label{gammalimit}
\end{align}
which could be intuitively denoted by $\mu^{\alpha-1} \dd \mu$,
that is obtained by taking the limit $\beta \rightarrow \infty$.
The posterior based on the prior $\pi_\alpha$ is invariant with respect to the time scale change.
The prior \eqref{gammalimit} is a natural generalization of the improper prior
$\prod_{i=1}^d (\mu_i^{\alpha_i-1} \dd \mu_i)$ discussed by \cite{K:AS2004}
for the finite-dimensional independent Poisson model.

Here, we consider a generalization of the gamma process prior \eqref{gammaprior}.
Let
\begin{align}
\pi_{\alpha,\beta,\gamma}(\dd \mu)
:= p^{\Di}_\alpha(\dd \overline{\mu}) \, p^\mathrm{Ga}_{|\alpha|-\gamma,\beta}(\masslambda) \, \dd \masslambda,
\label{extgammaprior}
\end{align}
where $\gamma < |\alpha|$.
We consider
\begin{align}
\label{ourpriormu}
 \pi_{\alpha,\gamma}(\dd \mu) := \pi_{\alpha,\beta=\infty,\gamma}(\dd \mu)
=  p^{\Di}_\alpha(\dd \overline{\mu}) \, p^\mathrm{Ga}_{|\alpha|-\gamma,\infty}(\masslambda) \, \dd \masslambda,
\end{align}
which is a generalization of \eqref{gammalimit},
by taking the limit $\beta \rightarrow \infty$.
We denote the distributions of $\lambda(u) = \int k(u,v) \mu(\dd v)$
corresponding to \eqref{extgammaprior} and \eqref{ourpriormu} by $\pi_{\alpha,\beta,\gamma}(\dd \lambda)$
and $\pi_{\alpha,\gamma}(\dd \lambda)$, respectively,
by abuse of notation without confusion.

We investigate Bayesian inference based on the prior $\pi_{\alpha,\beta,\gamma}$.
From \eqref{ldecomp} and \eqref{extgammaprior},
the posterior distribution of $\mu$ with respect to prior $\pi_{\alpha,\beta,\gamma}$
and observed data $\xx$ is proportional to
\begin{align}
\biggl\{ \prod_{i=1}^\xnum & k(x_i,\overline{\mu}) \biggr\}
p^{\Di}_\alpha(\dd \overline{\mu})
\frac{(s\masslambda)^\xnum}{\xnum!} \exp \bigl( -s\masslambda \bigr)
p^\mathrm{Ga}_{|\alpha|-\gamma,\beta}(\masslambda) \, \dd \masslambda \notag \\
&\propto\
\biggl\{ \prod_{i=1}^\xnum k(x_i,\overline{\mu}) \biggr\}
p^{\Di}_\alpha(\dd \overline{\mu})
p^\mathrm{Ga}_{|\alpha|-\gamma+\xnum,\beta/(s\beta+1)}(\masslambda) \, \dd \masslambda.
\label{posterior}
\end{align}
Thus, the posterior and the Bayesian predictive density for the kernel mixture models
have more complex forms than those for the simple Gamma--Poisson processes.

The posterior mean of $\blambda$ given $\xx$ is
\begin{align}
\blambda_{\alpha, \boldsymbol{x}}(y)
= \frac{\lambda_{\alpha, \beta, \gamma, \xx}(y)}{|\lambda_{\alpha, \beta, \gamma, \boldsymbol{x}}|}
=&\ \frac{\int k(y,\overline{\mu})
\bigl\{ \prod_{i=1}^\xnum k(x_i,\overline{\mu}) \bigr\} p^{\Di}_\alpha(\dd \overline{\mu}) }
{\int \bigl\{ \prod_{i=1}^\xnum k(x_i,\overline{\mu}) \bigr\} p^{\Di}_\alpha(\dd \overline{\mu}) },
\end{align}
not depending on $\beta$, $\gamma$, or $s$.
For precise treatment of quantities related to
posterior distributions based on disintegration, see e.g. \citet{AS:James2005}.

The posterior means of $w$ and $\lambda$ are
\begin{align}
\masslambda_{\absalpha-\gamma,\beta,\xnum}
= \lvert \lambda_{\alpha,\beta,\gamma, \boldsymbol{x}} \rvert
=&\ \frac{1}{s+1/\beta}(\absalpha - \gamma + \xnum)
\label{pomeanw}
\end{align}
and
\begin{align}
\lambda_{\alpha,\beta,\gamma, \boldsymbol{x}}
&= \masslambda_{\absalpha-\gamma,\beta,\xnum} \overline{\lambda}_{\alpha,\xx}
= \frac{1}{s+1/\beta}(\absalpha - \gamma + \xnum)
\blambda_{\alpha, \boldsymbol{x}},
\label{pomeanlambda}
\end{align}
respectively.
Since the posterior density of $w$ depends on $\alpha$ and $\gamma$ only through $\absalpha - \gamma$
and on $\xx=(N,x_1,\ldots,x_N)$ only through $\xnum$, we denote the posterior mean of $w$ by
$\masslambda_{\absalpha - \gamma,\beta,\xnum} := \lvert \lambda_{\alpha,\beta,\gamma, \boldsymbol{x}} \rvert$.

From \eqref{posterior},
the Bayesian predictive density is given by
\begin{align*}
p^\mathrm{NB}_{t/(t+s+1/\beta),\absalpha-\gamma+\xnum}&(\ynum) p_\alpha(y_1,\ldots,y_\ynum \mid \ynum, \xx),
\end{align*}
where
\begin{align}
p_\alpha(y_1,\ldots,y_\ynum \mid \ynum, \xx)
=&\ \frac{\int \bigl\{ \prod_{j=1}^\ynum k(y_j,\overline{\mu}) \bigr\}
\bigl\{ \prod_{i=1}^\xnum k(x_i,\overline{\mu}) \bigr\} p^{\Di}_\alpha(\dd \overline{\mu})}
{\int \bigl\{ \prod_{i=1}^\xnum k(x_i,\overline{\mu}) \bigr\} p^{\Di}_\alpha(\dd \overline{\mu})}.
\label{preddensity}
\end{align}
and
\[
p^\mathrm{NB}_{t/(s+t+1/\beta),\absalpha-\gamma+\xnum}(\ynum) =
\frac{\Gamma(\ynum+\xnum+\absalpha-\gamma)}{\ynum! \Gamma(\xnum+\absalpha-\gamma)}
\left(\frac{t}{t+s+1/\beta}\right)^\ynum
\left(\frac{s+1/\beta}{t+s+1/\beta}\right)^{\absalpha-\gamma+\xnum}.
\]
Here, $p^\mathrm{NB}_{t/(s+t+1/\beta),\absalpha-\gamma+\xnum}(\ynum)$
is the negative binomial distribution 
with success probability $t/(t+s+1/\beta)$ and number of failures $\absalpha-\gamma+\xnum$.
We can evaluate the posterior mean \eqref{pomeanlambda} of $\lambda$
and the predictive density \eqref{preddensity} using 
Markov chain Monte Carlo (MCMC) methods as discussed in Section \ref{numerical}.

By taking the limit $\beta \rightarrow \infty$,
we obtain the Bayes estimate and the Bayesian predictive density based on the improper prior $\pi_{\alpha,\gamma}$.
From \eqref{posterior}, the posterior with respect to the improper prior $\pi_{\alpha,\gamma}$ and observation $\xx$ is
proportional to
\begin{align*}
\biggl\{ \prod_{i=1}^\xnum & k(x_i,\overline{\mu}) \biggr\} 
p^{\Di}_\alpha(\dd \overline{\mu})
\frac{(s\masslambda)^\xnum}{\xnum!} \exp \bigl( -s\masslambda \bigr)
p^\mathrm{Ga}_{|\alpha|-\gamma,\infty}(\masslambda) \, \dd \masslambda \\
&\propto
\biggl\{ \prod_{i=1}^\xnum k(x_i,\overline{\mu}) \biggr\}
p^{\Di}_\alpha(\dd \overline{\mu})
p^\mathrm{Ga}_{|\alpha|-\gamma+\xnum,1/s}(\masslambda) \, \dd \masslambda.
\end{align*}

The posterior means of $\blambda$, $w$ and $\lambda$ with the improper prior $\pi_{\alpha,\gamma}$ are
\begin{align}
\blambda_{\alpha, \boldsymbol{x}}(y)
=&\ \frac{\int k(y,\overline{\mu})
\bigl\{ \prod_{i=1}^\xnum k(x_i,\overline{\mu}) \bigr\} p^{\Di}_\alpha(\dd \overline{\mu}) }
{\int \bigl\{ \prod_{i=1}^\xnum k(x_i,\overline{\mu}) \bigr\} p^{\Di}_\alpha(\dd \overline{\mu}) },
\label{posteriormean}
\end{align}
\begin{align}
\masslambda_{\absalpha-\gamma,\xnum} =&\ \lvert \lambda_{\alpha,\gamma, \boldsymbol{x}} \rvert
= \frac{1}{s} (\absalpha - \gamma + \xnum),
\label{pomeanw2}
\end{align}
and
\begin{align}
\lambda_{\alpha,\gamma, \boldsymbol{x}}
= \masslambda_{\absalpha-\gamma,\xnum} \blambda_{\alpha, \boldsymbol{x}}
= \frac{1}{s} (\absalpha - \gamma + \xnum) \blambda_{\alpha, \boldsymbol{x}},
\label{pomeanlambda2}
\end{align}
respectively.

The Bayesian predictive density with the improper prior $\pi_{\alpha,\gamma}$ is given by
\begin{align*}
p^\mathrm{NB}_{t/(t+s),\absalpha-\gamma+\xnum}&(\ynum) p_\alpha(y_1,\ldots,y_\ynum \mid \ynum, \xx),
\end{align*}
where
\begin{align*}
p^\mathrm{NB}_{t/(t+s),\absalpha-\gamma+\xnum}(\ynum) &=
\frac{\Gamma(\ynum+\xnum+\absalpha-\gamma)}{\ynum! \Gamma(\xnum+\absalpha-\gamma)}
\left(\frac{t}{t+s}\right)^\ynum
\left(\frac{s}{t+s}\right)^{\absalpha-\gamma+\xnum}.
\end{align*}

\subsection{Estimation as infinitesimal prediction}\label{infinitesimal prediction}
For the finite-dimensional independent Poisson model,
Bayesian parameter estimation under the Kullback--Leibler loss
is formulated as infinitesimal Bayesian prediction \citep{K:JMVA2006,K:JMVA15}.
We derive the corresponding results for nonhomogeneous Poisson processes.
Using this formulation,
an integral representation of the Kullback--Leibler risk of a predictive density is obtained.
This representation provides a unified framework for prediction and estimation and is a basis for later discussions.
Intuitively, for the nonhomogeneous Poisson model, estimation is infinitesimal prediction and
prediction is cumulative estimation.

Let
$\boldsymbol{z} = (\xnum+\ynum,x_1,\ldots,x_\xnum,y_1,\ldots,y_\ynum)$.
Then, the density of $\boldsymbol{z}$ is
\begin{align}
p_\lambda(\zz) &= \biggl\{\displaystyle \prod_{i=1}^{\xnum+\ynum} \blambda(z_i) \biggr\}
\frac{\{(s+t)\masslambda\}^{\xnum+\ynum}}{(\xnum+\ynum)!} \exp \bigl\{ - (s+t) \masslambda \bigr\}.
\end{align}
Since the conditional density
\begin{align}
p_\lambda(\boldsymbol{x} \mid \boldsymbol{z})
&= \frac{p_\lambda(\xx,\zz)}{p_\lambda(\zz)}
= \frac{p_\lambda(\xx)p_\lambda(\yy)}{p_\lambda(\zz)}
={\xnum+\ynum \choose \xnum} \biggl(\frac{s}{s+t}\biggr)^\xnum \biggl(\frac{t}{s+t}\biggr)^\ynum
\end{align}
does not depend on $\lambda$,
$\zz$ is a sufficient statistic when $\xx$ and $\yy$ are observed.
We denote $p_\lambda(\xx \mid \zz)$ by $p(\xx \mid \zz)$.
Let
$p_\pi(\xx, \zz) := \int p_\lambda(\xx, \zz) \pi(\dd \lambda)$ and
$p_\pi(\zz) := \int p_\lambda(\zz) \pi(\dd \lambda)$,
where $\pi(\dd \lambda)$ is a prior.
Then,
we have
\begin{align}
p_\pi(\yy \mid \xx) = p_\pi(\zz \mid \xx)
= \frac{p(\xx \mid \zz) p_\pi(\zz)}{p_\pi(\xx)}.
\label{eq:sufficiency1}
\end{align}

We consider a nonhomogeneous Poisson process $\zz_\tau = (N_\tau, z_1, \ldots, z_{N_\tau})$ distributed
according to $\Po(\tau \lambda)$ on $U$ with time $\tau \geq 0$.
When we need to explicitly specify time $\tau$, $z_i$ is denoted by $z_{\tau,i}$.
For $a,b > 0$, the difference between $\zz_{a+b} = (N_{a+b}, z_1, \ldots, z_{N_{a+b}})$
and $\zz_a = (N_a, z_1, \ldots, z_{N_a})$ is defined by
$\zz_{a+b} - \zz_a := (N_{a+b}-N_a, z_{N_a+1}, \ldots, z_{N_{a+b}})$.
For all $a,b > 0$, $\zz_a$ and $\zz_{a+b}-\zz_a$ are independently distributed according to $\Po(a \lambda)$
and $\Po(b \lambda)$,
respectively.
This spatial point process $\zz_\tau$ on $U$ is called a pure immigration process.
Then, the simultaneous distribution of $\xx$ and $\yy$ is identical to that of $\zz_s$ and $\zz_{s+t} - \zz_{s}$.
The probability density of $\zz_\tau$ is denoted by $p_{\lambda,\tau}(\zz_\tau)$.
From \eqref{eq:sufficiency1},
the Kullback--Leibler risk is represented by
\begin{align}
\E_\lambda \left[ \log \frac{p_\lambda(\yy \mid \xx)}{p_\pi(\yy \mid \xx)} \right]
&= \E_\lambda \left[ \log \frac{p_{\lambda,s+t}(\zz_{s+t})}{p_{\pi,s+t}(\zz_{s+t})} \right]
   - \E_\lambda \left[ \log \frac{p_{\lambda,s}(\zz_s)}{p_{\pi,s}(\zz_s)} \right] \notag \\
\label{KLintegral}
&= \int_s^{s+t} \frac{\partial}{\partial \tau}
\E_\lambda \left[ \log \frac{p_{\lambda,\tau}(\zz_\tau)}{p_{\pi,\tau}(\zz_\tau)} \right]
\dd \tau,
\end{align}
where $p_{\pi,\tau}(\zz_\tau) = \int p_{\lambda,\tau}(\zz_\tau) \pi(\dd \lambda)$.
Here, we assume that the expectation
$\E_\lambda \left[ \log \{p_{\lambda,\tau}(\zz_\tau)/p_{\pi,\tau}(\zz_\tau)\} \right]$
exists and is differentiable with respect to $\tau \in [s,s+t]$.

In order to evaluate the integrand of \eqref{KLintegral},
we prepare Lemma \ref{fielddifferential}, which is a generalization of
Lemma \ref{1dimdifferential} below,
which is essentially used in \citet{K:JMVA15}.
\begin{lemma}
\label{1dimdifferential}
Suppose that $N_\tau$ $(\tau \geq 0)$ is distributed according to a Poisson distribution with mean $\tau \masslambda$,
where $w$ is a fixed positive real number.
Let $h$ be a function from the nonnegative integers to the real numbers.
Assume that $\sum_{n=0}^\infty |h(n)| (\theta^n / n!) \exp(-\theta) < \infty$ for all $\theta > 0$.
Then,
\begin{align*}
\frac{\dd}{\dd \tau} \E \Bigl[ h(N_\tau) \Bigr]
&= \E \left[ \Bigl( \frac{N_\tau}{\tau}-\masslambda \Bigr) h(N_\tau) \right]
= \masslambda \, \E \Bigl[ h(N_\tau+1) - h(N_\tau) \Bigr].
\end{align*}
\end{lemma}

\vspace{0.4cm}

The proof is straightforward.
\begin{lemma}
\label{fielddifferential}
Suppose that
$\zz_\tau = (N_\tau,z_1,\ldots,z_{N_\tau})$
is distributed according to the nonhomogeneous Poisson process $\Po(\tau \lambda)$ $(\tau \geq 0)$.

Let $h_n$ $(n=0,1,2,\ldots)$ be functions from $U^n \times \mathbb{R}_{\geq 0}$ to $\mathbb{R}$,
where $U^0 := \emptyset$ and $h_0$ is a constant.
Every function $h_n(z_1,\ldots,z_n,\tau)$ is symmetric with respect to $z_1,\ldots,z_n$
and is differentiable with respect to $\tau \geq 0$ for every fixed $z_1,\ldots,z_n$.
Let $h$ be a function from $(\cup_{n=0}^\infty U^n) \times \mathbb{R}_{\geq 0}$ to $\mathbb{R}$
defined by $h(\zz,\tau) = h(n,z_1,\ldots,z_n,\tau) := h_n(z_1,\ldots,z_n,\tau)$, where $\zz=(n,z_1,\ldots,z_n)$.
Assume that $f(r,s) := \E [h(\zz_r,s)]$ exists and is differentiable with respect to $(r,s) \in \mathbb{R}^2_{>0}$
in a neighborhood of every $r=s>0$
and that $\frac{\partial}{\partial s}\E[h(\zz_r,s)] = \E[\frac{\partial}{\partial s} h(\zz_r,s)]$.

Then,
\begin{align}
\frac{\dd}{\dd \tau} \E \bigl[ h(\zz_\tau,\tau) \bigr]
=&\ \int_U \lambda(y) \E \bigl[ h(\zz_\tau + \delta_y) - h(\zz_\tau) \bigr] \dd y
+ \E \left[ \frac{\partial h}{\partial r} (\zz_\tau,r) \biggr|_{r=\tau} \right],
\label{PoissonFieldStein}
\end{align}
where $\zz_\tau + \delta_y := (N_\tau+1,z_1,\ldots,z_{N_\tau},y)$.
\end{lemma}

\vspace{0.4cm}
The proof is given in Appendix A.

Now, we have an integral representation of the risk of a Bayesian predictive density.
\vspace{0.4cm}
\begin{theorem}\label{PRF}
Suppose that
$\zz_\tau = (N_\tau,z_1,\ldots,z_{N_\tau})$
is distributed according to the nonhomogeneous Poisson process $\Po(\tau \lambda)$ $(\tau \geq 0)$.

Assume that $f(r,s) := \E [\log \{p_{s,\lambda}(\zz_r)/p_{s,\pi}(\zz_r)\}]$ exists
and is differentiable with respect to $(r,s) \in \mathbb{R}^2_{>0}$
in a neighborhood of every $r=s>0$
and that $\frac{\partial}{\partial s}\E[\log \{p_{s,\lambda}(\zz_r)/p_{s,\pi}(\zz_r)\}]
= \E[\frac{\partial}{\partial s} \log \{p_{s,\lambda}(\zz_r)/p_{s,\pi}(\zz_r)\}]$.

Then,
\begin{align}
\E_\lambda \left[ \log \frac{p_\lambda(\yy \mid \xx)}{p_\pi(\yy \mid \xx)} \right]
=&\ \int_s^{s+t} \E_\lambda \left[ w_{\pi,\zz_\tau,\tau} - w 
- w \log \frac{w_{\pi,\zz_\tau,\tau}}{w}
+ w \int_U \blambda (y) \log \frac{\blambda(y)}{\blambda_{\pi,\zz_\tau,\tau}(y)} \dd y \right] \dd \tau \notag \\
\label{integralrepresentation}
=&\ \int_s^{s+t} \E_\lambda \left[ \int_U \biggl\{
\lambda_{\pi,\zz_\tau,\tau}(y) - \lambda(y)
+ \lambda(y) \log \frac{\lambda(y)}{\lambda_{\pi,\zz_\tau,\tau}(y)} \biggr\} \dd y \right] \dd \tau.
\end{align}
\end{theorem}
\vspace{0.4cm}

The proof is given in Appendix A.

Since the integrand \eqref{integralrepresentation}
multiplied by $t$ coincides with the Kullback--Leibler loss \eqref{eq:KLeat} 
for the intensity estimator $\lambda_{\pi,\zz,\tau}$,
Bayes estimation under the Kullback--Leibler loss for the nonhomogeneous Poisson model
can be regarded as infinitesimal Bayesian prediction
as in the finite-dimensional independent Poisson model. 

When the prior is decomposable as $\pi(\dd \lambda) = \pi(\dd w) \pi(\dd \overline{\lambda})$,
the assumption in Theorem \ref{PRF} and Theorem \ref{domination} in the next subsection
can be easily verified in many problems.
If $\pi(\dd \lambda) = \pi(\dd w) \pi(\dd \overline{\lambda})$, then
\begin{align}
f(r,s)&=
\E_\lambda \left[
\log
\frac
{\prod^{N_r}_{i=1} \lambda (z_{r,i}) \exp ( -s \masslambda )} 
{\int \prod^{N_r}_{j=1} \tlambda (z_{r,j}) \exp ( -s \tmasslambda ) \pi ( \dd \tlambda)}
\right] \notag \\
&= \E_\lambda \left[
\log
\frac
{\prod^{N_r}_{i=1} \masslambda^{N_r} \exp ( -s \masslambda )} 
{\int \prod^{N_r}_{j=1} \tmasslambda^{N_r} \exp ( -s \tmasslambda ) \pi ( \dd \tmasslambda)}
\right]
+
\E_\lambda \left[
\log
\frac
{\prod^{N_r}_{i=1} \blambda (z_{r,i})} 
{\int \prod^{N_r}_{j=1} \btlambda (z_{r,j}) \pi ( \dd \btlambda)}
\right].
\label{decomposition}
\end{align}

The first term in \eqref{decomposition} can be explicitly evaluated for our priors.
The second term in \eqref{decomposition} is represented by
\begin{align}
h(r) := \E_\lambda \left[\log \biggl\{ 
\frac{\prod^n_{j=1} \overline{\lambda}(z_j)}
{\int \prod^n_{k=1} \overline{\lambda}_{\overline{\mu}} (z_k) \pi ( \dd \overline{\mu} )} \biggr\} \right]
= \exp (- rw) \sum^\infty_{n = 0} \frac{(rw)^n}{n!} c(n),
\label{power1}
\end{align}
where
\[
c(n) := \idotsint 
\biggl\{ \prod^n_{i=1} \overline{ \lambda } ( z_i) \biggr\}
\log \biggl\{ 
\frac{\prod^n_{j=1} \overline{\lambda}(z_j)}
{\int \prod^n_{k=1} \overline{\lambda}_{\overline{\mu}} (z_k) \pi ( \dd \overline{\mu} )} \biggr\}
\dd z_1 \dotsb \dd z_n.
\]
Here, $c(n)$ is the cumulative Kullback--Leibler risk when we predict
$z_1,\ldots,z_n$ independently distributed according to a probability density $\blambda$
using the simultaneous Bayesian predictive density
$\int \prod^{N_r}_{j=1} \btlambda (z_{r,j}) \pi ( \dd \btlambda)$.

Therefore, in order to verify the assumption in Theorem \ref{PRF}, it is sufficient 
to show that \eqref{power1} converges for every $0 \leq r < \infty$.

In many problems, we can verify that \eqref{power1} converges for every $0 \leq r < \infty$
using the inequality
\begin{align}
0 \leq c(n) \leq n \biggl\{ \sup_{\overline{\mu} \in B} D(\overline{\lambda},\overline{\lambda}_{\overline{\mu}})
+ \int_B \pi(\dd \overline{\mu}) \biggr\},
\label{resolv}
\end{align}
where $B$ is an arbitrary measurable set in the space of intensity functions.
The inequality \eqref{resolv} is used to define the index of resolvability
by \citet{B:BS99} \citep[see also][]{BC:IEEEIT91}.
Thus, a sufficient condition is that there exists $B$ such that
$\sup_{\overline{\mu} \in B} D(\overline{\lambda},\overline{\lambda}_{\overline{\mu}}) < \infty$ and $\pi(B) > 0$.
For example, the nonhomogeneous Poisson model with the gamma prior and without a kernel discussed in Subsection \ref{Kernel}
does not satisfy this condition.

If the power series of $r$ in \eqref{power1} converges for all $0 \leq r < \infty$, we have
\begin{align}
 \frac{ \dd}{ \dd r}  h(r)& = -w \exp ( -rw) \sum^\infty_{ n=0} \frac{ (rw)^n}{n!} c (n)
+ \exp ( -rw) \sum^\infty_{n=1} \frac{ w ( r w )^{n-1}}{ ( n -1) ! } c (n) \notag \\
&= w \exp ( -rw) \sum^\infty_{n=0}  \frac{ ( r w )^{n-1}}{n!} \{ c ( n+1) - c(n) \}
\label{power2} \\
&= w \E_\lambda \left[
\int \overline{\lambda}(y)
\log
\frac
{\overline{\lambda}(y)}
{\overline{\lambda}_{\pi,\mathbf{z}_r}(y)}
\dd y 
\right], \notag
\end{align}
because
\begin{align}
c(n+1) - c(n)
&= \int \overline{ \lambda } (y)  \idotsint \Bigl\{ \prod^n_{i=1} \overline{\lambda}(z_i) \Bigr\}
\log
\left\{ 
\frac
{\overline{\lambda}(y) \prod^n_{ i=1} \overline{ \lambda } ( z_i ) }
{\int \overline{\lambda}_{\overline{\mu}} (y) \prod^n_{ i=1} \overline{ \lambda }_{ \overline{\mu} } (z_i)
\pi (\dd \overline{\mu})}
\right\}   
\dd z_1 \dotsb \dd z_n \dd y \notag \\
&\ \ \ \, -  \idotsint
\Bigl\{ \prod^n_{ i=1}  \overline{ \lambda }(z_i) \Bigr\}
 \log \left\{  
 \frac{\prod^n_{i=1}\overline{\lambda}(z_i)}
 {\prod^n_{i=1} \int \overline{ \lambda }_{ \overline{\mu} } (z_i) \pi(\dd \overline{\mu})}
 \right\}
\dd z_1 \dotsb \dd z_n \notag \\
&= \idotsint \Bigl\{ \prod^n_{ i=1} \overline{ \lambda } (z_i) \Bigr\}
\int \overline{ \lambda } (y)
\log \frac
{\overline{ \lambda } (y)}
{\overline{ \lambda }_{ \pi , z^n} (y)}
\dd y \dd z_1 \dotsb \dd z_n,
\label{individual}
\end{align}
where
\begin{align*}
\overline{ \lambda }_{\pi , z^n} (y)
= \frac{\int \prod^n_{ i=1} 
\overline{ \lambda }_{ \overline{\mu} } (z_i)
\overline{ \lambda }_{ \overline{\mu} } (y) \pi ( \dd  \overline{\mu} ) }
{\int \prod^n_{j=1} 
\overline{ \lambda }_{ \overline{\mu} } (z_j) \pi ( \dd  \overline{\mu} )}.
\end{align*}
Thus, the difference \eqref{individual} between $c(n+1)$ and $c(n)$ is the individual Kullback--Leibler risk
of the Bayesian predictive density for predicting $z_{n+1}$ using $z_1,\ldots,z_n$,
where $z_1,\ldots,z_{n+1}$ are independently distributed according to probability density $\blambda$.
The power series of $r$ in \eqref{power1} converges for $0 \leq r < \infty$
if and only if the power series in \eqref{power2} converges for $0 \leq r < \infty$.

\subsection{Shrinkage priors dominating the prior $\pi_\alpha$ and their admissibility}

We propose shrinkage priors dominating $\pi_\alpha(\dd \mu) = \mu^{\alpha-1} \dd \mu$ when $|\alpha| > 1$.
We prove admissibility of Bayes estimators and Bayesian predictive densities based on the shrinkage priors.

Although few studies on admissibility concerning infinite-dimensional models
have been carried out,
we show that Blyth's method with a convex loss is useful for our infinite-dimensional problem
because the method works even when the support of the prior is a small subset of the whole parameter space.
The key idea is to decompose the problem into two sub-problems:
a one-dimensional problem with an improper prior, and an infinite-dimensional problem with a proper prior.

\vspace{0.4cm}

\begin{theorem}\label{domination}~~
Suppose that
$\zz_\tau = (N_\tau,z_1,\ldots,z_{N_\tau})$
is distributed according to the nonhomogeneous Poisson process $\Po(\tau \lambda)$ $(\tau \geq 0)$.

Assume that $f(r,s) := \E [\log \{p_{s,\lambda}(\zz_r)/p_{s,\pi}(\zz_r)\}]$ exists and is differentiable with respect to $(r,s) \in \mathbb{R}^2_{>0}$
in a neighborhood of every $r=s>0$
and that $\frac{\partial}{\partial s}\E[\log \{p_{s,\lambda}(\zz_r)/p_{s,\pi}(\zz_r)\}]=
\E[\frac{\partial}{\partial s} \log \{p_{s,\lambda}(\zz_r)/p_{s,\pi}(\zz_r)\}]$.
\begin{enumerate}
\item 
When $|\alpha|-\gamma > 1$,
the Bayesian predictive density $p_{\alpha, \gamma}(\yy \mid \xx)$
based on
$\pi_{\alpha,\gamma}(\dd \mu)$
is dominated by the Bayesian predictive density $p_{\alpha, \tgamma}(\yy \mid \xx)$
based on $\pi_{\alpha, \tgamma}(\dd \mu)$, 
where $\tgamma := |\alpha| -1$.
\vspace{0.3cm}
\item When $|\alpha|-\gamma > 1$,
the generalized Bayes estimator of $\lambda$
based on
$\pi_{\alpha,\gamma}(\dd \mu)$
is dominated by the generalized Bayes estimator of $\lambda$
based on $\pi_{\alpha, \tgamma}(\dd \mu)$,
where $\tgamma := |\alpha| -1$.
\end{enumerate}
\end{theorem}
\vspace{0.4cm}

The proof is given in Appendix A.

In particular,
if $|\alpha| > 1$, the Bayesian predictive density based on
$\pi_{\alpha}(\dd \mu) = \pi_{\alpha,\gamma=0}(\dd \mu) = \mu^{\alpha-1} \dd \mu$
is dominated by the Bayesian predictive density based on $\pi_{\alpha,\tgamma=|\alpha|-1}(\dd \mu)$.

\vspace{0.4cm}
\begin{lemma}\label{1dimPoisson}
Let $X$ be a Poisson random variable with mean $\theta \geq 0$, and let $c$ be an arbitrary positive constant.
Then,
\begin{align*}
\theta \, \E_\theta& \biggl[
\log (X+1+c) - \log (X+1) \biggr]
\leq c - c \exp(-\theta) < c.
\end{align*}
\label{Poissonineq}
\end{lemma}

\vspace{0.4cm}

The proof of Lemma \ref{1dimPoisson}, which is used in \citet{K:AS2004}, is given in Appendix A. for self-containedness.

Next, we prove admissibility of Bayes estimators and Bayesian predictive densities based on our shrinkage priors.

Suppose that the parameter space $\Lambda$ is a set of finite measures
that are mutually absolutely continuous with respect to the Lebesgue measure on $U \subset \mathbb{R}^d$.
We assume that if $\lambda = (w,\overline{\lambda}) \in \Lambda$, then $(w',\overline{\lambda}) \in \Lambda$ for all $w' > 0$.
Thus, $\Lambda = \mathbb{R}^+ \times \overline{\Lambda}$,
where $\mathbb{R}^+$ is the set of positive real numbers
and $\overline{\Lambda}$ is a set of probability densities
that are mutually absolutely continuous with respect to the Lebesgue measure.

Note that the support of a prior for the kernel mixture model only covers a small subset
of the whole parameter space $\Lambda$ if $\Lambda$ is a large set
as in ordinary settings for nonparametric inference.

Let $\mathcal{A}$ be the space of all finite measures on $U$ such that the measure is mutually absolutely continuous
with respect to the Lebesgue measure.
For estimation of $\lambda$, we choose an intensity estimate from $\mathcal{A}$.
Let $\mathcal{P}$ be the space of all probability measures
on $\cup_{m=0}^\infty U^m$
such that the marginal probability $P(m)$ of $P \in \mathcal{P}$
is positive (i.e.\;not equal to $0$) for every nonnegative integer $m$ and
the conditional probability $P(\cdot \mid m)$ on $U^m$ has density with respect to
the Lebesgue measure on $U^m \subset \mathbb{R}^m$ for every $m$.
For prediction problem, we choose a probability density from $\mathcal{P}$.

We assume that $\int k(y,u) \overline{\mu}(\dd u) \in \overline{\Lambda}$ with probability 1
if $\bmu$ is distributed according to $\Di(\alpha)$.
This condition is naturally satisfied in ordinary settings for nonparametric intensity estimation.

\vspace{0.4cm}

\begin{theorem}\label{admissible}~~~~
\begin{enumerate}
 \item 
The Bayesian predictive density based on
$\pi_{\alpha,\gamma}(\dd \mu)
:=p^{\Di}_\alpha(\dd \overline{\mu}) \, p^\mathrm{Ga}_{|\alpha|-\gamma,\infty}(\masslambda) \dd \masslambda$
is admissible under the Kullback--Leibler loss if $|\alpha| - 1 \leq \gamma < |\alpha|$.
\vspace{0.3cm}
\item
The generalized Bayes estimator $\lambda_{\alpha,\gamma,\xx}$ based on
$\pi_{\alpha,\gamma}(\dd \mu)
:=p^{\Di}_\alpha(\dd \overline{\mu}) \, p^\mathrm{Ga}_{|\alpha|-\gamma,\infty}(\masslambda) \dd \masslambda$
is admissible under the Kullback--Leibler loss if $|\alpha| - 1 \leq \gamma < |\alpha|$.
\end{enumerate}
\end{theorem}

\vspace{0.4cm}

The proof is given in Appendix A.

In particular,
if $|\alpha| > 1$, then the Bayesian predictive density based on $\pi_{\alpha,\gamma=|\alpha|-1}$
dominating that based on $\pi_{\alpha} = \pi_{\alpha,\gamma=0}$ is admissible.

\section{Numerical evaluation of predictive densities and estimators}\label{numerical}
In this section, we explore numerical methods by which to evaluate Bayesian predictive densities and Bayes estimates.

In the posterior \eqref{posterior} with respect to the prior
$\pi_{\alpha,\beta,\gamma}(\dd \mu) :=  p^{\Di}_\alpha(\dd \overline{\mu}) \, p^\mathrm{Ga}_{|\alpha|-\gamma,\beta}(\masslambda) \dd w$,
$w$ and $\overline{\mu}$ are independently distributed.
The posterior distribution of $w$ is $p^\mathrm{Ga}_{|\alpha|-\gamma+\xnum,\beta/(s\beta+1)}(\masslambda) \, \dd \masslambda$.
The posterior distribution of $\overline{\mu}$ is proportional to a Dirichlet process mixture
\begin{align}
\int & \biggl\{ \prod_{l=1}^\xnum \int k(x_l,u_l) \overline{\mu}(\dd u_l) \biggr\} p_\alpha^{\Di} (\dd \overline{\mu})
= \idotsint \biggl\{ \prod_{l=1}^\xnum k(x_l,u_l) \biggr\}
p^\mathrm{CR}_{\alpha}(\dd u_1,\ldots, \dd u_\xnum),
\label{CRrepresentation}
\end{align}
where
\begin{align*}
p^\mathrm{CR}_{\alpha}(\dd u_1,\ldots, \dd u_\xnum)
= \int \Bigl\{ \prod_{l=1}^\xnum \overline{\mu}(\dd u_l) \Bigr\} p_\alpha^{\Di} (\dd \overline{\mu})
\end{align*}
is the distribution of the Chinese restaurant process with measure $\alpha$ 
\citep[e.g.][p.~193]{Phadia16}.
If random variables $u_1,u_2,\ldots,u_\xnum$ are
distributed according to
a Chinese restaurant process $P^{\mathrm{CR}}_\alpha(\dd u_1, \ldots, \dd u_\xnum)$,
then they are sequentially distributed according to
\begin{align*}
P^{\mathrm{CR}}_\alpha(\dd u_1)
= \frac{1}{|\alpha|} \alpha(\dd u_1)
\end{align*}
and
\begin{align*}
P^{\mathrm{CR}}_\alpha(\dd u_{k+1} \mid u_1,\ldots,u_k)
= \frac{1}{|\alpha| + k} \Bigl\{ \alpha(\dd u_{k+1})
+ \sum_{i=1}^k \delta_{u_i} (\dd u_{k+1})
\Bigr\} \hspace{0.5cm} (k=1,2,\ldots,\xnum-1),
\end{align*}
see e.g.\ \citet{Phadia16}. 

We can numerically evaluate quantities concerning the posterior density of $\lambda$ such as
the Bayes estimate of $\lambda$ and Bayesian predictive density of $\yy$ given $\xx$
using MCMC based on the representation \eqref{CRrepresentation}.
Various MCMC methods for density estimation
\citep[e.g.][]{Phadia16} based on nonparametric Bayes with Dirichlet priors can be applied
to our problem.

From \eqref{posteriormean}, the posterior mean $\overline{\lambda}$ with respect to
$\pi_{\alpha,\beta,\gamma}$ and observation $\xx$ is
\begin{align}
\overline{\lambda}_{\alpha,\xx}
&=
\frac{\int k(y,\overline{\mu}) \{ \prod_{l=1}^\xnum k(x_l,\overline{\mu}) \} p_\alpha^{\Di} (\dd \overline{\mu})}
{\int \{ \prod_{l=1}^\xnum k(x_l,\overline{\mu}) \} p_\alpha^{\Di} (\dd \overline{\mu})} \notag \\
&=
\frac{\idotsint k(y,u_{\xnum+1}) \{\prod_{l=1}^\xnum k(x_l,u_l)\}
p^\mathrm{CR}_{\alpha}(\dd u_1,\ldots, \dd u_\xnum, \dd u_{\xnum+1})}
{\idotsint \{ \prod_{l=1}^\xnum k(x_l,u_l)\}
p^\mathrm{CR}_{\alpha}(\dd u_1,\ldots,\dd u_\xnum)} \notag \\
&=
\frac{\idotsint k(y,u_{\xnum+1}) p^\mathrm{CR}_{\alpha}(\dd u_{\xnum+1} \mid u_1,\ldots,u_\xnum) \{\prod_{l=1}^\xnum k(x_l,u_l)\}
p^\mathrm{CR}_{\alpha}(\dd u_1,\ldots, \dd u_\xnum)}
{\idotsint \{ \prod_{l=1}^\xnum k(x_l,u_l)\}
p^\mathrm{CR}_{\alpha}(\dd u_1,\ldots, \dd u_\xnum)},
\label{posmeanlambdabar}
\end{align}
which does not depend on $\beta$ or $\gamma$.
Thus, the posterior mean of $\overline{\lambda}$ with respect to the prior
$\pi_{\alpha,\beta,\gamma}$ is $\overline{\lambda}_{\alpha,\xx}$,
which does not depend on $\beta$ or $\gamma$.

We can obtain the Bayes estimates of $\lambda$ based on $\pi_{\alpha,\beta,\gamma}$
and $\pi_{\alpha,\gamma} := \pi_{\alpha,\beta=\infty,\gamma}$
using \eqref{pomeanw}, \eqref{pomeanlambda}, \eqref{pomeanw2}, \eqref{pomeanlambda2}, and
\eqref{posmeanlambdabar}.

\vspace{0.4cm}

\noindent
Example.
We consider intensity functions on a unit circle $[0,2\pi)$
for simplicity.
Let the true intensity function be
\[
 \lambda(u) = \sin(u) + 2.
\]

We obtain Bayes estimates of $\lambda$ based on priors $\pi_{\alpha=1,\beta=\infty,\gamma=0}$
and $\pi_{\alpha=1,\beta=\infty,\gamma=|\alpha|-1}$.
Then, $|\alpha| := \int_0^{2\pi} \alpha(u) \dd u = 2 \pi > 1$.
We adopt the von Mises kernel 
\begin{align*}
 k_\kappa(x;\mu) = \frac{\exp \{\kappa \cos (x-\mu) \}}{2 \pi I_0(\kappa)}
\end{align*}
with $\kappa = 5$.
The observation time length $s$ is set to be $1$.

In Figure 1, Bayesian estimates of $\lambda$ based on non-shrinkage prior $\pi_{\alpha=1,\beta=\infty,\gamma=0}$
and shrinkage prior $\pi_{\alpha=1,\beta=\infty,\gamma=|\alpha|-1}$
with respect to observation $
(0.29, 1.55, 2.06, 2.85, 2.87, 3.60, 5.55, 5.61, 5.65, 6.01)$ are shown.
The Bayesian estimates based on the non-shrinkage prior and the shrinkage prior are 
$\lambda_{\alpha = 1,\gamma = 0, \boldsymbol{x}}
= (\xnum + 2 \pi) \blambda_{\alpha, \boldsymbol{x}}$,
and
$\lambda_{\alpha = 1,\gamma = \absalpha-1, \boldsymbol{x}}
= (\xnum + 1) \blambda_{\alpha, \boldsymbol{x}}$,
respectively,
where $\blambda_{\alpha, \boldsymbol{x}}$ is numerically evaluated by MCMC.

From Theorem \ref{domination},
the estimator based on the shrinkage prior dominates that based on the non-shrinkage prior.
Specifically, the Kullback--Leibler risk of the estimator based on the shrinkage prior is smaller
than that of the estimator based on the non-shrinkage prior for all $\lambda$.
This example illustrates how shrinkage priors improve Bayes estimators based on non-shrinkage priors.

\begin{figure}[!htbp]
\begin{center}
\includegraphics[height=8cm]{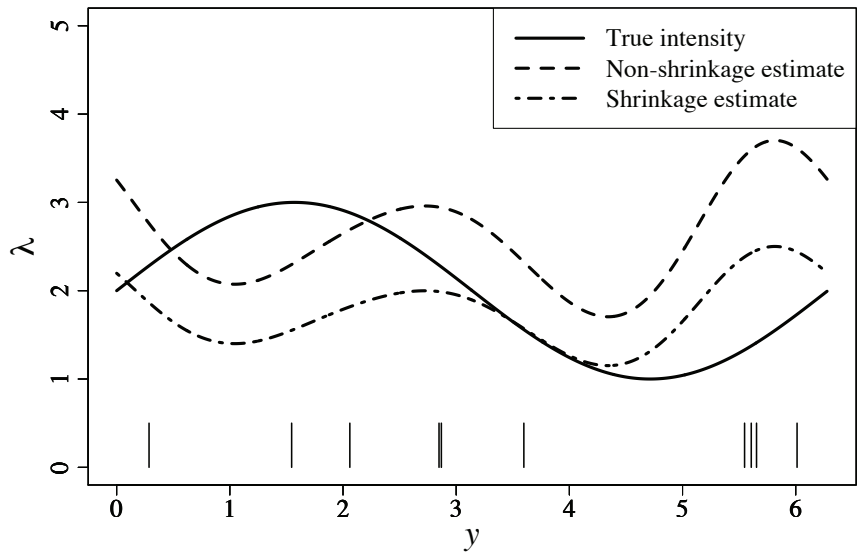}
\end{center}
\caption{Bayesian estimates of $\lambda$ based on non-shrinkage prior $\pi_{\alpha=1,\beta=\infty,\gamma=0}$
and shrinkage prior $\pi_{\alpha=1,\beta=\infty,\gamma=|\alpha|-1}$.
The (sorted) observed points indicated by the vertical ticks at the bottom
are located at $0.29, 1.55, 2.06, 2.85, 2.87, 3.60, 5.55, 5.61, 5.65$, and $6.01$.}
\end{figure}

\section{Discussion and Conclusion}

A class of improper shrinkage priors for nonparametric Bayesian inference
of the nonhomogeneous Poisson processes with kernel mixture is investigated.
The Bayesian predictive densities and the Bayes estimators based on the priors are admissible
under the Kullback--Leibler loss.
The class of improper  priors could be useful as objective priors for nonhomogeneous Poisson models.

The information theoretic properties of the Kullback--Leibler loss play essential roles in our theory.
Our results can be easily generalized to various methods for intensity estimation
including models based on kernels with parameters.

Although few studies on admissibility of Bayesian prediction and inference based on improper priors
for infinite-dimensional models have been carried out,
our approach could be useful for various infinite-dimensional problems.

One important direction of future research is the generalization of the present results
for nonhomogeneous Poisson process models
to those for other stochastic process models such as nonhomogeneous negative binomial process models.
For finite dimensional models, \citet{HK:MMS2019,HK:JMVA2020} extended the theory for finite dimensional Poisson models
to finite dimensional negative binomial models and negative multinomial models.
These generalizations require techniques not used in the theory for the Poisson models.

Hazard rate models closely related to nonhomogeneous Poisson models are widely used in applications and kernel mixtures with general random
measures have been investigated \citep{AS:James2005, AS:DBPP2009}.
The use of various random measures other than the gamma random measure has been studied.
The relation between theories on hazard rate models with general random measures
and the present study is also an important topic for future study.

Nonhomogeneous Poisson process models have a variety of applications.
For example, they have recently been used in neural information processing \citep{SK17,SK20}.
The development of data analyzing methods for such applications based on the theory presented here is also a future challenge.
In applications such as survival analysis, weighted gamma processes are often used as priors.
By extending the theorems in \citet{K:JMVA15} to those for infinite-dimensional models,
we could generalize our results to improve the limits of weighted gamma process priors.

Furthermore, estimation and prediction with infinite dimensional statistical models based on Brownian motions
could be studied in line with the approach adopted in the present paper.

\vspace{0.4cm}

\noindent
{\bf Acknowledgments.}
The author is grateful to two reviewers for their helpful and constructive comments.
This research was supported by MEXT KAKENHI Grant Number 16H06533,
JST CREST Grant Number JPMJCR1763,
and AMED Grant Numbers JP20dm0207001 and JP20dm0307009.

\vspace{0.4cm}

\bibliographystyle{sjs}

\vspace{0.4cm}

\appendix

\section{Proofs of Theorems and Lemmas}

\begin{proof}[Proof of Lemma \ref{fielddifferential}]
Since 
\[
f(r,s) := \E \bigl[ h(\zz_r,s) \bigr]
= \sum_{n=0}^\infty 
\left[ \frac{(r w)^n}{n!} \exp (-r w) \idotsint \Bigl\{ \prod_{i=1}^n \blambda (z_i) \Bigr\} h_n(z_1,\ldots,z_n,s)
\dd z_1 \cdots \dd z_n
\right]
\]
and $\frac{\partial}{\partial s}\E[h(\zz_r,s)] = \E[\frac{\partial}{\partial s} h(\zz_r,s)]$,
we have
\begin{align*}
\frac{\dd}{\dd \tau} \E \bigl[h(\zz_\tau,\tau)\bigr]
=&\ \frac{\dd}{\dd \tau} f(\tau,\tau)
= \frac{\partial}{\partial r} f(r,\tau) \biggr|_{r=\tau} + \frac{\partial}{\partial s} f(\tau,s) \biggr|_{s=\tau} \\
=&\
w \sum_{n=1}^\infty \left[
\frac{(\tau w)^{n-1}}{(n-1)!} \exp (-\tau w) \idotsint \Bigl\{ \prod_{i=1}^n \blambda (z_i) \Bigr\} h_n(z_1,\ldots,z_n,\tau)
\dd z_1 \cdots \dd z_n
\right] \\
& 
- w
\sum_{n=0}^\infty
\left[\frac{(\tau w)^n}{n!} \exp (-\tau w) \idotsint \Bigl\{ \prod_{i=1}^n \blambda (z_i) \Bigr\} h_n(z_1,\ldots,z_n,\tau)
\dd z_1 \cdots \dd z_n
\right] \\
&+ \frac{\partial}{\partial r} \E \Bigl[ h(\zz_\tau,r) \Bigr] \biggr|_{r=\tau} \\
=&\ \int_U \lambda(y) \E_\lambda \Bigl[ h(\zz_\tau + \delta_y,\tau) - h(\zz_\tau,\tau) \Bigr] \dd y
+ \E \left[ \frac{\partial h}{\partial r}(\zz_\tau,r) \biggr|_{r=\tau} \right].
\end{align*}
\end{proof}

\vspace{0.4cm}

\begin{proof}[Proof of Theorem \ref{PRF}]
Let
\begin{align*}
h(\zz_r,s) &:= \log \frac{p_{s,\lambda}(\zz_r)}{p_{s,\pi}(\zz_r)}
= \log \frac
{\displaystyle \Bigl\{ \prod^{N_r}_{i=1} s\lambda (z_{r,i}) \biggr\}
\frac{1}{N_r!} \exp ( -s \masslambda )}
{\displaystyle \int \Bigl\{ \prod^{N_r}_{i=1} s \tlambda (z_{r,i}) \Bigr\}
\frac{1}{N_r!} \exp ( -s \tmasslambda ) \pi ( \dd \tlambda)}
\\
&= - 
\log \int 
\Bigl\{ \prod^{N_r}_{i=1} \frac{\tlambda (z_{r,i})}{\lambda (z_{r,i})} \Bigr\}
\exp ( -s \tmasslambda ) \pi ( \dd \tlambda)
- s \masslambda.
\end{align*}
From Lemma \ref{fielddifferential},
the integrand of \eqref{KLintegral} is represented by
\begin{align}
\frac{\dd}{\dd \tau}&
\E_\lambda \left[ \log \frac{p_{\tau,\lambda}(\zz_\tau)}{p_{\tau,\pi}(\zz_\tau)} \right]
= \frac{\dd}{\dd \tau}
\E_\lambda \Bigl[ h(\zz_\tau,\tau) \Bigr]
= - \frac{\dd}{\dd \tau}
\E_\lambda \left[
\log \int 
\Bigl\{ \prod^{N_\tau}_{i=1} \frac{\tlambda (z_{\tau,i})}{\lambda (z_{\tau,i})} \Bigr\}
\exp ( -\tau \tmasslambda ) \pi ( \dd \tlambda) \right]
- w \notag \\
= & - \int_U \lambda (y) \E_\lambda
\left[ \log \int \Bigl\{ \prod^{N_\tau}_{i=1} \frac{\tlambda (z_{\tau,i})}{\lambda (z_{\tau,i})} \Bigr\}
\frac{\tlambda (y)}{\lambda (y)} \exp ( -\tau \tmasslambda ) \pi ( \dd \tlambda)  
 - \log \int \Bigl\{ \prod^{N_\tau}_{i=1} \frac{\tlambda (z_{\tau,i})}{\lambda (z_{\tau,i})} \Bigr\}
\exp ( -\tau \tmasslambda ) \pi ( \dd \tlambda) \right] \dd y \notag \\
  &- \E_{\lambda}
\left[ \frac{\displaystyle  - \int \tmasslambda \bigl\{ \textstyle\prod\limits^{N_\tau}_{i=1} \tlambda (z_{\tau,i}) \bigr\}
\exp ( -\tau \tmasslambda ) \pi ( \dd \tlambda ) }
{\displaystyle \int \bigl\{ \textstyle\prod\limits^{N_\tau}_{i=1} \lambda' (z_{\tau,i}) \bigr\}
\exp ( - \tau \tmasslambda ) \pi ( \dd \tlambda ) } \right]
- w \notag \\
= & - \int_U \lambda (y) \log \frac{\lambda_{\pi,\zz_{\tau},\tau} (y)}{\lambda (y)} \dd y
+ \masslambda_{\pi,\zz_{\tau},\tau} - \masslambda \notag \\
=&\ w_{\pi,\zz_\tau,\tau} - w 
- w \log \frac{w_{\pi,\zz_\tau,\tau}}{w}
+ w \int_U \blambda (y) \log \frac{\blambda(y)}{\blambda_{\pi,\zz_\tau,\tau}(y)} \dd y.
\label{differential}
\end{align}
From \eqref{KLintegral} and \eqref{differential}, we have the desired result.
\end{proof}

\vspace{0.4cm}

\begin{proof}[Proof of Theorem \ref{domination}]
We prove the first part (1) for Bayesian predictive densities.
The second part (2) for Bayes estimators can be proved in a similar manner.

Since \eqref{posteriormean}, the posterior mean of
$\overline{\lambda}$ based on $\pi_{\alpha,\gamma}$ coincides with that based on $\pi_{\alpha,\widetilde{\gamma}}$.
Therefore, from Theorem \ref{PRF}, we have 
\begin{align}
\E_\lambda & \Bigl[ D(p(\yy \mid \lambda), p_{\alpha,\gamma}(\yy \mid \xx)) \Bigr]
- \E_\lambda \Bigl[ D(p(\yy \mid \lambda), p_{\alpha,\tgamma}(\yy \mid \xx)) \Bigr] \notag \\
&= \int_s^t w \,
\E_\lambda \left[ \frac{w_{\alpha,\gamma,\tau}}{w} - 1 - \log \frac{w_{\alpha,\gamma,\tau}}{w} \right] \dd \tau
- \int_s^t w \,
\E_\lambda \left[ \frac{w_{\alpha,\tgamma,\tau}}{w} - 1 - \log \frac{w_{\alpha,\tgamma,\tau}}{w} \right] \dd \tau,
\label{riskdiff1}
\end{align}
where $w_{\alpha,\gamma,\tau} := w_{\alpha,\gamma,\zz(\tau),\tau}$.
From \eqref{pomeanw2},
the posterior means of $w$ with respect to priors $\pi_\alpha=\pi_{\alpha,\gamma=0}$
and $\pi_{\alpha,\widetilde{\gamma}=|\alpha|-1}$ are
$w_{\alpha,\tau} = (N_\tau + |\alpha|)/\tau$
and
$w_{\alpha,\tgamma=|\alpha|-1,\tau} = (N_\tau + 1)/\tau$,
respectively, when we observe $N_\tau$.

Thus, the integrand of \eqref{riskdiff1} with $\gamma=0$ and $\tgamma=|\alpha|-1$ is
\begin{align}
w \E_\lambda & \biggl[
\frac{w_{\alpha,\tau}}{w} -1 - \log \frac{w_{\alpha,\tau}}{w} \biggr]
- w \E_\lambda \biggl[
\frac{w_{\alpha,\tgamma=|\alpha|-1,\tau}}{w} - 1
- \log \frac{w_{\alpha,\tgamma=|\alpha|-1,\tau}}{w} \biggr] \notag \\
=&\
\frac{|\alpha|}{\tau} - w \, \E_\lambda \left[ \log \frac{N_\tau + |\alpha|}{\tau} \right]
- \frac{1}{\tau} + w \, \E_{\lambda} \left[ \log \frac{N_\tau+1}{\tau} \right] \notag \\
=&\ \frac{1}{\tau} \biggl\{ |\alpha|-1 - \tau w \, \E_\lambda \Bigl[\log (N_\tau + |\alpha|)\Bigr]
+ \tau w \, \E_\lambda \Bigl[\log (N_\tau + 1) \Bigr] \biggr\}.
\label{integrand}
\end{align}
From Lemma 3 below, \eqref{integrand} is positive.
\end{proof}

\vspace{0.4cm}
\noindent
\begin{proof}[Proof of Lemma \ref{Poissonineq}]
\begin{align*}
c - \sum_{n=0}^\infty& \exp(-\theta) \frac{\theta^{n+1}}{n!} \biggl\{
\log (n+1+c) - \log (n+1) \biggr\} \\
\geq&\ c - \sum_{n=0}^\infty \exp(-\theta) \frac{\theta^{n+1}}{n!} \frac{c}{n+1}
= c \exp(-\theta) \biggl\{ \exp \theta - \sum_{n=0}^\infty \frac{\theta^{n+1}}{(n+1)!} \biggr\}
= c \exp(-\theta) > 0.
\end{align*}
\end{proof}

\vspace{0.4cm}

\begin{proof}[Proof of Theorem \ref{admissible}]
We prove the first part (1) for Bayesian predictive densities using Blyth's method with a convex loss.
Blyth's method is widely used to prove admissibility usually for problems with a finite-dimensional parameter space.
The method with a convex loss also works for our infinite-dimensional problem.
The second part (2) for Bayes estimators can be proved in a similar manner.

By the assumption, the distributions $p_\lambda(\xx)$ $(\lambda \in \Lambda)$ are absolutely continuous with
respect to each other.
The action space $\mathcal{P}$ is convex because
$a q(\yy) + (1-a) q'(\yy) \in \mathcal{P}$ if both $q(\yy)$ and $q'(\yy)$ belong to $\mathcal{P}$ and $0 \leq a \leq 1$.
The Kullback--Leibler loss
\begin{align}
L(\lambda, q) :=
\E_{\lambda} \left[ \log \frac{p_\lambda(\ynum,y_1,\ldots,y_\ynum)}{q(\ynum,y_1,\ldots,y_\ynum)} \right]
\notag
\end{align}
is a strictly convex function from $q \in \mathcal{P}$ to $[0,\infty]$ for every fixed $\lambda$.

We use a monotonically increasing sequence of proper priors defined by
$\pi_{\alpha, \gamma}^{[l]}(\dd \mu) = \pi_{\alpha,\gamma}(\dd \mu) \frac{1}{2}h_l^2(w)$ $(l=1,2,\ldots)$,
where
$\pi_{\alpha,\gamma}(\dd \mu)
=  p^{\Di}_\alpha(\dd \overline{\mu}) \, p^\mathrm{Ga}_{|\alpha|-\gamma,\infty}(\masslambda)
 \dd \masslambda
=  p^{\Di}_\alpha(\dd \overline{\mu}) \, \masslambda^{|\alpha|-\gamma-1}
 \dd \masslambda$
and
\begin{align*}
h_l(w) &= \begin{cases}
                 1 & \mbox{~if~~~} 0 \leq w \leq 1 \\
                 \displaystyle 1 - \frac{\log w}{\log l} & \mbox{~if~~~} 1 < w \leq l \\
                 0 & \mbox{~if~~~} l < w.
                \end{cases}
\end{align*}
Function sequences of this kind were introduced by \citet{BH:SDT06} and are used to
prove admissibility of linear estimators \citep{GY:AS1988}
and admissibility of predictive densities \citep{K:AS2004} for finite-dimensional independent Poisson models.
Then, $\pi_{\alpha,\gamma}^{[l]} \ll \pi_{\alpha,\gamma}$ for all $l=1,2,\ldots$.
Let $C := \{(w,\overline{\lambda}) \mid 0 < w < 1, \overline{\lambda} \in \overline{\Lambda} \}$.
Then, $\pi(C) > 0$ and $\dd \pi^{[l]}/\dd \pi = 1/2$ if $\mu \in C$.

In order to prove admissibility of $p_{\alpha, \gamma}(\yy \mid \xx)$ 
based on $\pi_{\alpha, \gamma}$ with $0< |\alpha| - \gamma \leq 1$,
it suffices to show
\begin{align}
\lim_{l \rightarrow \infty} \int \pi_{\alpha, \gamma}^{[l]}(\dd \lambda) &
\Bigl\{ \E_\lambda \bigl[ D(p_\lambda(\yy),p_{\alpha, \gamma}(\yy \mid \xx)) \bigr]
- \E_\lambda \bigl[ D(p_\lambda(\yy),p_{\alpha, \gamma}^{[l]}(\yy \mid \xx)) \bigr] \Bigr\} = 0,
\label{Blyth}
\end{align}
where $p^{[l]}$ is the Bayesian predictive density based on $\pi_{\alpha,\gamma}^{[l]}$.
The proof of the sufficiency of \eqref{Blyth} for admissibility is given in Appendix B.

We set $c = \gamma - |\alpha| + 1$ $(0 \leq c < 1)$, and $g_l(w) = (1/2) h_l^2 (w)$.
From Theorem \ref{PRF} and \eqref{riskdiff1}, we have the expression
\begin{align}
&\int \pi_{\alpha, \beta}^{[l]}(\dd \lambda)
\Bigl\{ \E_\lambda \bigl[ D(p(\yy \mid \lambda),p_{\alpha, \beta}(\yy \mid \xx)) \bigr]
 - \E_\lambda \bigl[ D(p(\yy \mid \lambda),p_{\alpha, \beta}^{[l]}(\yy \mid \xx)) \bigr] \Bigr\} \notag \\
&= \int^{s+t}_s
\biggl\{
\int_0^\infty g_l(w) w^{-c} \sum_{n=0}^\infty \exp(-\tau w) \frac{(\tau w)^n}{n!}
\biggl( \frac{n+1-c}{\tau} - \widehat{w}_{l,\tau}(n) - w \log \frac{n+1-c}{\tau \widehat{w}_{l,\tau}(n)} \biggr) \dd w
\biggr\} \dd \tau,
\label{Blyth2}
\end{align}
where
\begin{align}
\widehat{w}_{l,\tau}(n)
&= \frac{\int_0^\infty w \exp (- \tau w) \frac{(\tau w)^{n}}{n!} w^{-c} g_l(w) \dd w}
{\int_0^\infty \exp (- \tau w) \frac{(\tau w)^{n}}{n!} w^{-c} g_l(w) \dd w}
= \frac{\int_0^\infty \exp (- \tau w) (\tau w)^{n+1-c} g_l(w) \dd w}
{\tau \int_0^\infty \exp (- \tau w) (\tau w)^{n-c} g_l(w) \dd w} \nonumber \\
&= \frac{n + 1 -c}{\tau}
+ \frac{\int_0^\infty \exp (- \tau w) (\tau w)^{n+1-c} g'_l(w) \dd w}
{\tau^2 \int_0^\infty \exp (- \tau w) (\tau w)^{n-c} g_l(w) \dd w},
\label{estimator}
\end{align}
because $\overline{\lambda}_{\alpha,\gamma} = \overline{\lambda}^{[l]}_{\alpha,\gamma}$ and it does not depend on $\gamma$.

The integral \eqref{Blyth2} coincides with (22) in \citet{K:AS2004} and
converges to $0$ as $l$ goes to infinity as shown in the following.

We show that \eqref{Blyth2} converges to $0$ in the same manner as the evaluation of (22) in \citet{K:AS2004}.
We include the proof here for self-containedness.

We have
\begin{align}
\int_0^\infty & g_l(w) w^{-c} \sum_{z=0}^\infty \exp(-\tau w) \frac{(\tau w)^z}{z!}
\biggl( \frac{z + 1 - c}{\tau} - \widehat{w}_{l,\tau} - w \log \frac{z+1-c}{\tau \widehat{w}_{l,\tau}} \biggr) \dd w \nonumber \\
\leq & \int_0^\infty g_l(w) w^{-c} \sum_{z=0}^\infty \exp(-\tau w) \frac{(\tau w)^z}{z!}
\biggl\{ \frac{z + 1 - c}{\tau} - \widehat{w}_{l,\tau} + w \frac{\tau \widehat{w}_{l,\tau} - (z+1-c)}{z+1-c} \biggr\} \dd w \nonumber \\
=& \sum_{z=0}^\infty \frac{\tau^{c-1}}{z!} \biggl\{
\int_0^\infty \exp(-\tau w) g_l(w) \frac{(\tau w)^{z+1-c}}{z+1-c} \dd w
- \int_0^\infty \exp(-\tau w) g_l(w) (\tau w)^{z-c} \dd w
\biggr\} \notag \\
&\cdot \{ \tau \widehat{w}_{l,\tau} - (z+1-c) \} \nonumber \\
=& \sum_{z=0}^\infty \frac{\tau^{c-2}}{z!}
\int_0^\infty \exp(-\tau w) g'_l(w) \frac{(\tau w)^{z+1-c}}{z+1-c} \dd w
\cdot \{ \tau \widehat{w}_{l,\tau} - (z+1-c) \}.
\label{inequality1}
\end{align}
Using (\ref{estimator}) and the inequality
$(z+1)/(z+1-c) \leq 1/(1-c)$, where $0 \leq c < 1$, we have
\begin{align}
\int_0^\infty & g_l(w) w^{-c} \sum_{z=0}^\infty \exp(-\tau w) \frac{(\tau w)^z}{z!}
\biggl( \frac{z + 1 - c}{\tau} - \widehat{w}_{l,\tau}
- w \log \frac{z+1-c}{\tau \widehat{w}_{l,\tau}} \biggr) \dd w \nonumber \\
\leq&\
\frac{\tau^{c-3}}{1-c} 
\sum_{z=0}^\infty \frac{1}{(z+1)!}
\frac{\displaystyle \biggl\{ \int_0^\infty \exp(-\tau w) (\tau w)^{z+1-c} g'_l(w) \dd w \biggr\}^2}
{\int \exp (- \tau w) (\tau w)^{z-c} g_l(w) \dd w} \nonumber \\
=&\ \frac{\tau^{c-3}}{1-c} \sum_{z=0}^\infty \frac{2}{(z+1)!}
\frac{\displaystyle \biggl\{ \int_0^\infty \exp(-\tau w) (\tau w)^{z-c} (\tau w h'_l(w)) h_l(w) \dd w \biggr\}^2}
{\int \exp (- \tau \bmu) (\tau w)^{z-c} h_l^2(w) \dd w} \nonumber \\
\leq&\
\frac{\tau^{c-3}}{1-c}
\sum_{z=0}^\infty \frac{2}{(z+1)!}
\frac{\displaystyle \biggl\{ \int_0^\infty \exp(-\tau w) (\tau w)^{z-c} (\tau w h'_l(w))^2 \dd w \biggr\}
\cdot \biggl\{ \int_0^\infty \exp(-\tau w) (\tau w)^{z-c} h_l^2(w) \dd w \biggr\}}
{\int \exp (- \tau w) (\tau w)^{z-c} h_l^2(w) \dd w} \nonumber \\
=&\
\frac{\tau^{c-3}}{1-c} 
\sum_{z=0}^\infty \frac{2}{(z+1)!}
\int_0^\infty \exp(-\tau w) (\tau w)^{z+2-c} (h'_l(w))^2 \dd w \nonumber \\
=&\
\frac{2\tau^{c-3}}{1-c} \int^\infty_0 \{1-\exp(-\tau w)\} (\tau w)^{1-c} (h'_l (w))^2 \dd w
\leq
\frac{2}{(1-c)\tau^2} \int^\infty_0 w^{1-c} (h'_l (w))^2 \dd w
\label{inequality2}
\end{align}
The derivative of $h_l(w)$ is
\begin{align}
h'_l(w) = \begin{cases}
                 0 & \mbox{~if~~~} 0 < w < 1 \\
                 {\displaystyle - \frac{1}{w \log l}} & \mbox{~if~~~} 1 < w < l \\
                 0 & \mbox{~if~~~} l < w.
                \end{cases}
\label{h_l derivative}
\end{align}
From (\ref{Blyth2}), (\ref{inequality2}) and (\ref{h_l derivative}), we have
\begin{align*}
\int & \pi_{\alpha, \beta}^{[l]}(\dd \lambda) \E_\lambda \Bigl[ D(p_\lambda(\yy),p_{\alpha, \beta}(\yy \mid \xx))
- D(p_\lambda(y), p_{\alpha, \beta}^{[l]}(\yy \mid \xx)) \Bigr] \\
&\leq \int_s^{s+t} \biggl\{ \frac{2}{(1-c)\tau^2} \int^\infty_0 w^{1-c} (h_l' (w))^2 \dd w \biggr\} \dd \tau
= \int^{s+t}_s \biggl\{ \frac{2}{(1-c)\tau^2} \int_1^l \frac{1}{w^{1+c}(\log l)^2} \dd w \biggr\} \dd \tau \\
&\leq \int_s^{s+t} \frac{2}{(1-c)\tau^2} \frac{1}{\log l} \dd \tau
= \frac{2}{(1-c) \log l} \biggl( \frac{1}{s} - \frac{1}{s+t} \biggr)  \rightarrow 0 \ \ \mbox{as} \ \ l \rightarrow \infty.
\end{align*}
\end{proof}

\section{Sufficiency of \eqref{Blyth} for admissibility}

We show that \eqref{Blyth} is sufficient for admissibility.
The proof for our infinite-dimensional problem
parallels the proof for finite-dimensional problems with a convex loss
\citep[see][Chapter 3]{Schervish95}.

We denote
a map from $\xx$ to a probability distribution of $\yy$ by $q$,
and denote the probability density that is the image of $\xx$ under $q$ by $q_{\xx}$.
The map $q$ corresponds to a predictive density.

Let $L(\lambda,q_{\xx})$ be a loss function, which is strictly convex with respect to $q_{\xx}$,
and let $R(\lambda,q):=\E_\lambda[L(\lambda,q_{\xx})]$,
where $\xx$ is distributed according to a probability density $p_\lambda$.
Assume that a predictive density $q$ satisfying
\begin{align}
\lim_{l \rightarrow \infty}
\left\{ R(\pi^{[l]}, q) - R(\pi^{[l]}, q^{[l]}) \right\} = 0,
\label{sufficiency condition}
\end{align}
where $q^{[l]}$ is the Bayesian predictive density based on $\pi^{[l]}$ and
$R(\pi^{[l]}, q^{[l]}) := \int R(\lambda, q^{[l]}) \pi^{[l]}(\dd \lambda)$
is the Bayes risk of $q^{[l]}$ with respect to the prior $\pi^{[l]}$,
is inadmissible.
Assume that $\dd \pi^{[l]}/\dd \pi \geq a > 0$ if $\mu \in C$ for a subset $C$ of the parameter space.
We assume that all probability measures in $\{p_\lambda:\lambda \in \Lambda\}$
are mutually absolutely continuous with respect to each other.

In our problem, \eqref{sufficiency condition} corresponds to \eqref{Blyth},
$L$ is the Kullback--Leibler loss, $\pi^{[l]} = \pi^{[l]}_{\alpha,\gamma}$,
and $\dd \pi^{[l]}/\dd \pi = 1/2$ if
$\mu \in C := \{(w,\overline{\lambda}) \mid 0 < w < 1, \overline{\lambda} \in \overline{\Lambda} \}$.

We prove admissibility of the predictive density $q$ by contradiction.
Since $q$ is inadmissible, there exists another predictive density $q'$ such that, for all $\lambda \in \Lambda$,
$R(\lambda,q') \leq R(\lambda,q)$,
and for at least one parameter value $\lambda_0 \in \Lambda$,
$R(\lambda_0,q') < R(\lambda_0,q)$.
Since $L(\lambda,q_{\xx})$ is strictly convex with respect to $q_{\xx}$ for every $\lambda$,
\begin{align*}
L(\lambda,q''_{\xx}) \leq \frac{1}{2} \bigl\{L(\lambda,q'_{\xx}) + L(\lambda,q_{\xx})\bigr\},
\end{align*}
where $q'' := (q + q')/2$,
for every $\lambda$ and $\xx$.
For every $\xx$ such that $q_{\xx} \neq q'_{\xx}$, we have
\begin{align*}
 L(\lambda,q''_{\xx}) < \frac{1}{2} \bigl\{L(\lambda,q'_{\xx}) + L(\lambda,q_{\xx}) \bigr\}.
\end{align*}
Here, $P_{\lambda_0}(q'_{\xx} \neq q_{\xx}) > 0$
because otherwise $R(\lambda_0,q') < R(\lambda_0,q)$ does not hold.
Thus, $P_\lambda(q'_{\xx} \neq q_{\xx}) > 0$ for every $\lambda$
because all probability measures $p_\lambda$ $(\lambda \in \Lambda)$ are mutually absolutely continuous
with respect to each other.
Thus, for every $\lambda$,
$R(\lambda,q'') < R(\lambda,q)$.
Therefore, for all $l=1,2,3,\ldots$,
\begin{align*}
R(\pi^{[l]}, q) - R(\pi^{[l]}, q^{[l]}) &\geq R(\pi^{[l]}, q) - R(\pi^{[l]}, q'')
\geq \int_C \left\{ R(\lambda,q) - R(\lambda,q'') \right\} \pi^{[l]}(\dd \lambda) \\
&\geq a \int_C \left\{ R(\lambda,q) - R(\lambda,q'') \right\} \pi(\dd \lambda)
> 0,
\end{align*}
where the second inequality is because $C$ is a subset of the parameter space $\Lambda$.
This contradicts \eqref{sufficiency condition}.
\end{document}